\documentclass[11pt,reqno, a4paper]{amsart}
\usepackage{amsmath,amsfonts,amsthm,amsopn,color,amssymb,enumitem}
\usepackage{palatino}
\usepackage{graphicx}
\usepackage{caption}
\usepackage{subcaption}
\usepackage[colorlinks=true]{hyperref}
\usepackage{geometry}
\hypersetup{urlcolor=blue, citecolor=red, linkcolor=blue}

\usepackage{esint}
\usepackage{pgfplots}
\usepackage{mathrsfs}
\usetikzlibrary{arrows}

\newtheorem{theorem}{Theorem}[section]

\newtheorem{proposition}[theorem]{Proposition}
\newtheorem{remark}[theorem]{Remark}

\newcommand{\R}{\mathbb{R}}

\newcommand{\D}{\mathbb{D}}

\def\bbm[#1]{\mbox{\boldmath $#1$}}
\newcommand{\beq }{\begin{equation}}
\newcommand{\eeq }{\end{equation}}

\def\sideremark#1{\ifvmode\leavevmode\fi\vadjust{\vbox to0pt{\vss% the remark3
 \hbox to 0pt{\hskip\hsize\hskip1em%                          will appear only
 \vbox{\hsize3cm\tiny\raggedright\pretolerance10000%          on the side
  \noindent #1\hfill}\hss}\vbox to8pt{\vfil}\vss}}}%

\begin{document}

\title[Vectorial $p$-Laplacian system]{Regularity and symmetry results for the vectorial p-Laplacian }

\author{Luigi Montoro, Luigi Muglia, Berardino Sciunzi, Domenico Vuono}

\email[Luigi Montoro]{luigi.montoro@unical.it}%
\email[Luigi Muglia]{luigi.muglia@unical.it}
\email[Berardino Sciunzi]{sciunzi@mat.unical.it}
\email[Domenico Vuono]{domenico.vuono@unical.it}
\address[L. Montoro, L. Muglia, B. Sciunzi, D. Vuono]{Dipartimento di Matematica e Informatica, Università della Calabria,
Ponte Pietro Bucci 31B, 87036 Arcavacata di Rende, Cosenza, Italy}



\maketitle

\begin{abstract}
We obtain some regularity  results for solutions to vectorial $p$-Laplace equations  $$ -{\boldsymbol \Delta}_p{\boldsymbol u}=-\operatorname{\bf div}(|D{\boldsymbol u}|^{p-2}D{\boldsymbol u}) = {\boldsymbol f}(x,{\boldsymbol u})\,\, \mbox{ in $\Omega$}\,.$$
More precisely we address the issue of second order estimates for the stress field. As a consequence of our regularity results we deduce a weighted Sobolev inequality that leads to  weak comparison principles. As a corollary we run over the moving plane technique to deduce symmetry and monotonicity results for the solutions, under suitable assumptions. 
\end{abstract}
\section{Introduction}

Let $\Omega$ be a bounded smooth domain in $\mathbb{R}^n$ with $n\geq 2$. In this paper  we will consider problems involving the vectorial operator $-{\boldsymbol \Delta}_p{\boldsymbol u}$ defined, for  smooth functions, by 
\beq\label{pLaplace}
-{\boldsymbol \Delta}_p{\boldsymbol u}=-\operatorname{\bf div}(|D{\boldsymbol u}|^{p-2}D{\boldsymbol u})
\eeq
where $p>1$ and $D{\boldsymbol u}$ is the Jacobian of  the  vector field ${\boldsymbol u}\,:\, \Omega \rightarrow \mathbb{R}^N$, $N\geq 2$. We shall use the notation  ${\boldsymbol u}=(u^1,\ldots,u^N)$, for a weak $C^{1}(\overline{\Omega})$ solution to the $p$-Laplace system 
\beq\label{system1}\tag{$p$-$S$}
\begin{cases}
-\operatorname{\bf div}(|D{\boldsymbol u}|^{p-2}D{\boldsymbol u})
=   {\boldsymbol f}(x,\boldsymbol u)& \mbox{in $\Omega$}\\
{\boldsymbol u}= 0  & \mbox{on  $\partial \Omega$}, 
\end{cases}
\eeq
where  $ {\boldsymbol f}:\overline\Omega \times \mathbb{R}^N\to\mathbb{R}^N$ satisfies a suitable set of assumptions \eqref{hpf} stated below.\\

\noindent The first step of our study is about  the regularity of the solutions. This issue is still undertaken in the literature although it represents an important milestone.
 It is well known  that the solutions to \eqref{system1} lie in $C^{1,\alpha}(\overline{\Omega})\cap C^2(\overline \Omega\setminus Z_{\boldsymbol{u}})$, where $Z_{\boldsymbol{u}}$ is the set where the  gradient vanishes, for every $p>1$ under suitable assumptions on the source term.  We refer the reader to  \cite{ChenDiB, KuuMin}  and to the  recent deep developments in  \cite{DM,DM2,DM3}.
  Second-order estimates were established in \cite{BaCiDiMa,Cma} for solutions to the $p$-Laplace system with right-hand side in $L^2$. In particular, in \cite{BaCiDiMa} the authors proved that $|D\boldsymbol{u}|^{p-2}D\boldsymbol{u}\in W^{1,2}_{loc}(\Omega)$, for $p>2(2-\sqrt{2})$.  \\In \cite{M}, the author proves that for $p\geq 3$ and $f\in W^{1,p'}(\Omega)$ then $D u$ locally lies in a suitable fractional Nikol'skii space. This result is obtained proving that  $|D\boldsymbol{u}|^{s-1}D\boldsymbol{u}\in W^{1,2}_{loc}(\Omega)$ for every $s\in ((p-1)/2,p/2]$ when $p\geq 3$.  As a rule,  if $p\in[2,3)$,  $|D\boldsymbol{u}|^{s-1}D\boldsymbol{u}\in W^{1,2}_{loc}(\Omega)$ for $s\in (1,p/2]$ follows too.\\
We extend the regularity estimates for the second derivatives of solutions to \eqref{system1} and,  as a consequence,  we also deduce integrability proprieties of $|D\boldsymbol{u}|^{-1}$. 

\noindent Our first result is the following:
\begin{theorem}\label{teosecond}
	Let $\Omega$ be a bounded  smooth domain and $p>1$. Let $\boldsymbol{u}\in C^1({\Omega})$ be a weak solution of \eqref{system1}, with $$\boldsymbol f(x,\boldsymbol u):=\boldsymbol f(x)\in W_{loc}^{1,\frac{n}{n-\gamma-s}}(\Omega)\cap C_{loc}^{0,\beta}(\Omega),$$
 where $0<s<n-\gamma$ and $\gamma<n-2$ ($\gamma=0$ if $n=2$). 
 
If $1<p\leq 2$ let us assume $0\leq \alpha <p-1$ and if $p>2$ let  $\alpha\in [0,1)$. Then,  for any $\tilde{\Omega} \subset\subset \Omega$,  it follows that 
	\begin{equation}\label{der2}
		\int_{\tilde{\Omega}\setminus Z_{\boldsymbol u}}\frac{|D\boldsymbol u|^{p-2-\alpha}\|D^2\boldsymbol{u}\|^2}{|x-y|^{\gamma}}dx\leq C,
	\end{equation}
	uniformly for any $y\in \tilde{\Omega}$, with  $C=C(\boldsymbol f, n,p,\alpha,\gamma,\tilde \Omega)$.
\end{theorem}

 We also point out  that Theorem \ref{teosecond} holds without any sign assumption on the source term $\boldsymbol{f}$. If else we assume that $\boldsymbol{f}$ has a sign, as a consequence of the previous result and we obtain the integrability properties of the inverse of the weight $|D \boldsymbol u|^{p-2}$.
\begin{theorem}\label{Reg2}
	Let $\Omega\subset \mathbb{R}^n$ be a  bounded smooth domain and let $\boldsymbol{u} \in C^{1}({\Omega})$ be a  weak solution of \eqref{system1} with $\boldsymbol f(x,\boldsymbol u):=\boldsymbol f(x)\in W_{loc}^{1,\frac{n}{n-\gamma-s}}(\Omega)\cap C_{loc}^{0,\beta}(\Omega)$, where $0<s<n-\gamma$ and $\gamma <n-2$ if $n\geq 3$ ($\gamma=0$ if $n=2$). Suppose that for some $i$:
	\[
	f^i\geq \tau >0 \quad in \quad  \Omega\,.
	\]
	Then, for any $\tilde{\Omega} \subset\subset \Omega$
	\begin{equation}\label{zicurli}
	    \int_{\tilde \Omega}\frac{1}{|D \boldsymbol u|^{\sigma}}\frac{1}{|x-y|^\gamma}\leq C,
	\end{equation} for any 
	\begin{equation}\label{eq:varsig}
		\sigma < \begin{cases}2p-3 &\text{if } 1<p\leq 2\\
		p-1 &\text{if } p>2,
		\end{cases}
	\end{equation} 
and $C=C(\boldsymbol f,n, p, \gamma, \sigma,\tau, \tilde \Omega)$ is a positive constant. In particular the Lesbegue measure of $Z_{\boldsymbol{u}}$ is zero.
\end{theorem}
\begin{remark}\label{eq:avetecapitochecisnio}
The reader will observe in the proof of Theorems \ref{teosecond}-\ref{Reg2} that the case $\gamma=0$ follows assuming $\boldsymbol  f(x)\in W_{loc}^{1,1}(\Omega)\cap C_{loc}^{0,\beta}(\Omega)$.
\end{remark}
\begin{remark}
    If $\boldsymbol{u}\in C^1(\bar\Omega)$ and $Z_{\boldsymbol{u}}\subset\subset\Omega$, it follows that the estimate \eqref{zicurli} holds up to the boundary, i.e.
    \begin{equation}\label{zicurli2}
	    \int_{ \Omega}\frac{1}{|D \boldsymbol u|^{\sigma}}\frac{1}{|x-y|^\gamma}\leq C.
	\end{equation}
\end{remark}
\noindent Note that Theorem \ref{Reg2} will be crucial in our application since it will allow us to deduce a weak comparison principle in small domains. 
\\

\noindent Let us now discuss our regularity result. The natural Calderon-Zygmund theory developed in \cite{BaCiDiMa} states that $|D\boldsymbol{u}|^{(p-2)}\|D^2\boldsymbol{u}\|\in L^2_{loc}(\Omega)$ if $p>2(2-\sqrt{2})$ and $\boldsymbol f\in L^2_{loc}(\Omega)$. A global result is also proved in \cite{BaCiDiMa}.
This implies that if $\boldsymbol f\in L^2(\Omega)$, $|D\boldsymbol{u}|^{(p-2)}D\boldsymbol{u}\in W^{1,2}(\Omega)$ when $p>2(2-\sqrt{2})$. 
Hidden in this result and under suitable assumptions,  $\boldsymbol{u}\in W^{2,2}_{loc}(\Omega)$ if $2(2-\sqrt 2)<p\leq 2$.\\ 

\noindent 
As an important consequence, exploiting our Theorem \ref{teosecond} and Theorem \ref{Reg2}, we recover that  $\boldsymbol{u}\in W^{2,2}_{loc}(\Omega)$ if $p\in(1,3)$ and, under the assumption that $\boldsymbol{f}$ has a sign, we prove that for $p\geq 3$,  $\boldsymbol u\in W^{2,q}_{loc}(\Omega)$ with $1\leq q<\frac{p-1}{p-2}$, see   Theorem \ref{Reg} in Section~\ref{seconda}.\\

\noindent The second part of the paper is devoted to obtain monotonicity and symmetry results, exploiting the moving plane procedure.  This is strongly related to weak and strong comparison principles that are known to fail, in general, in the vectorial case.  Therefore we first derive a comparison principle for our system, under very general assumptions on $\boldsymbol{f}$. We start the following

\noindent 
\begin{theorem}[Weak Comparison Principle in small domains]\label{thm:weak1}
Let $p>1$ and  $\boldsymbol u, \boldsymbol v\in C^1(\overline\Omega)$   such that either $\boldsymbol u$ or $\boldsymbol v$ is solution to \eqref{system1} and
\begin{equation}\label{eq:(20)0}
-\boldsymbol \Delta_p \boldsymbol u -\boldsymbol f(x,\boldsymbol u) \leq -\boldsymbol \Delta_p\boldsymbol v- \boldsymbol f(x, \boldsymbol v) \quad \mbox{in $\Omega$},\end{equation} where $\boldsymbol f$ satisfies \eqref{hpf}.
Let $\tilde \Omega \subset \subset \Omega$ be open and assume that
\begin{equation}\label{ipotesibrut}
    (u^i-v^i)(u^j-v^j)\geq 0 \quad \text{in } \tilde \Omega, \quad \text{for } i\neq j,
\end{equation}
and $\boldsymbol u\leq \boldsymbol v$ on $\partial \tilde\Omega$. \\
Then there exist $\lambda=\lambda(\boldsymbol f,p,N, \|D {\boldsymbol u}\|_{\infty})>0$ such that, if $|\tilde \Omega |<\lambda$, it follows that 
\[\boldsymbol u\leq \boldsymbol v \quad \text{in }  \tilde\Omega.\]
If $\tilde \Omega  \subseteq \Omega$, the same result holds assuming, for  $p>2$, that $Z_{\boldsymbol{u}}\subset\subset \Omega$.
\end{theorem}

We recall that the scalar case $N=1$ has been extensively treated in \cite{DamPa} for $1<p<2$, and subsequently extended in \cite{DamSci}, for every $p>1$.
For the vectorial case $N>1$ no general symmetry results are available in the literature and some achievement 
was obtained in \cite{MRS}, under stronger assumptions on the vectorial function $\boldsymbol{f}$ (namely assuming that  $\boldsymbol{f}$ is strictly increasing respect to $x$) and in the singular case $1<p<2$. We emphasize that the strict monotonicity assumption on the nonlinearity  in  \cite{MRS} prevents to consider the natural case of autonomous equations. \\

\noindent We conclude the paper deducing a symmetry and monotonicity result. To get this result we need further structural assumptions on the nonlinear term $\boldsymbol  f$.  For the convenience of the reader,  we denote by $({h_f})$ the following
\begin{enumerate}[label=$({h_f})$, ref=${h_f}$]
\item \label{hpf}
\begin{itemize}
\item[$(i)$] ${\boldsymbol f}(x, t_1,\ldots,t_N)$ is locally Lipschitz continuous and ${\boldsymbol f}(x,\cdot)$ is a locally Lipschitz continuous vector field, uniformly with respect to $x$, that is, each component $\ell\in\{1,\ldots,N\}$ \[f^\ell(x,t_1,\ldots, t_N):  \overline\Omega\times \mathbb{R}^N\rightarrow \mathbb R\] is a Lipschitz function with respect variable to $t_j$ namely, for every  $\Omega'\subseteq \overline\Omega$ and for every $M>0$,  there is a  positive constant $L_{\ell}=L_{\ell}(M,\Omega')$ such that for every $x \in \Omega'$ and every $ t_j,s_j \in [0,M]$ it holds:
$$  \vert f^\ell(x,t_1,\ldots,t_j,\ldots, t_N)- f^\ell(x,t_1,\ldots,s_j,\ldots, t_N)\vert \leq L_\ell \vert t_j-s_j \vert. $$

\item [$(ii)$] ${\boldsymbol f}$ is a nonnegative vector field and there exists $\ell\in \{1,...,N\}$ such that
 \[f^\ell(x,t_1,\ldots,t_N)>0\] 
for all $x\in \Omega$ and for every $t_j>0$. 
\item[$(iii)$] ${\boldsymbol f} (x,\cdot)$ satisfies   
$$\frac{\partial f^\ell}{\partial t_k}(x,t_1,\ldots,t_j,\ldots, t_N)\geq 0, \quad \text{a.e.  and for }  k\neq l.$$\end{itemize}
\end{enumerate}
We succeed in adapting the moving plane procedure in the context of vectorial $p$-Laplace equations proving the following:
\begin{theorem}\label{teoremaprincipale}
Let $\Omega$ be a bounded smooth domain of $\mathbb R^n$, 
convex with respect to the $x_1$-direction and symmetric with respect to $T_0$, where
\begin{equation*}%\label{eq:T0}
 T_0=\{x\in \Omega : x_1=0\}.
 \end{equation*}
Let ${\boldsymbol u}$  be a positive $C^1(\overline\Omega)$ weak solution to \eqref{system1} such that
\begin{equation}\label{eq:sugnufiutu}
|D\boldsymbol{u}|^{p-2}D\boldsymbol{u}\in W_{loc}^{1,2}(\Omega),
\end{equation}
 with ${\boldsymbol f}$ satisfying  \eqref{hpf} and
\begin{equation}\label{eq:fcresc}
\boldsymbol f(x_1,x_2,\ldots,x_n,\boldsymbol u)< \boldsymbol f(y_1,x_2,\ldots,x_n,\boldsymbol u) \quad \mbox{ where } x_1<y_1<0, 
\end{equation}
\begin{equation*}%\label{eq:fsym}
\boldsymbol f(-x_1,\ldots,x_{n-1},x_n,\boldsymbol u)=\boldsymbol f(x_1,\ldots,x_{n-1},x_n,\boldsymbol u).
\end{equation*}
We set $\displaystyle a:=\inf_{x\in \Omega}x_1$ and $\displaystyle b:=\sup_{x\in \Omega}x_1.$ Assume that 
\begin{equation}\label{ipotesisuisupporti}
    (u^i(x)-u^i(2\lambda-x_1,...,x_n)) (u^j(x)-u^j(2\lambda-x_1,...,x_n))\geq 0 \quad i,j=1,...,N,
\end{equation}
for any $a<\lambda<b.$\\
\noindent Then $\boldsymbol u$ is symmetric with respect to~$T_0$ and non-decreasing with respect to the $x_1$-direction in~$\Omega_0=\{x\in \Omega \, :\, x_1<0\}$.
\end{theorem}
\begin{remark}
Hypothesis \eqref{eq:sugnufiutu} is quite natural. Indeed it follows e.g. from \eqref{eq:w12} of Theorem \ref{Reg} or using \cite{BaCiDiMa, Cma}.
\end{remark}
Our paper is structured as follows. 
In Section~\ref{seconda} we prove the local regularity results given by Theorem \ref{teosecond} and by Theorem \ref{Reg2}.
In Section \ref{comparison} we prove the Symmetry result, i.e. Theorem  \ref{teoremaprincipale}.
\section{Regularity results for second derivative}\label{seconda}
\subsection*{Notation}\label{Notazioni}
Generic fixed numerical constants will be denoted by $C$ (with subscript in some case) and will be allowed to vary within a single line or formula. Moreover $f^+$  will stand for the positive a part of a function, i.e.
$f^+=\max\{f,0\}$. We also denote $|A|$ the Lebesgue measure of the set $A$.

We will use the bold style to stress the vectorial nature of  different quantities. For instance, a $N$-vectorial function $\boldsymbol{w}$ defined in $\Omega$ will be written as
$$
\boldsymbol{w}(x)=\left(w^1(x),\ldots,w^N(x) \right),
$$
where $w^\ell$ is a scalar function defined in $\Omega$ for $\ell=1,\ldots,N$.

\

We say that a vector field ${\boldsymbol u}$ weakly  solves \eqref{system1} if and only if ${\boldsymbol u} \in W^{1,p}_{0}(\Omega):=W^{1,p}_{0}(\Omega;\mathbb{R}^N)$ and 
\beq\label{weak}
\int_{\Omega} |D{\boldsymbol u}|^{p-2}D {\boldsymbol u} : D {\boldsymbol \varphi}\, dx = \int_{\Omega} \langle{\boldsymbol f}, \boldsymbol{\varphi}\rangle\, dx, \quad \forall  {\boldsymbol \varphi} \in W^{1,p}_{0}(\Omega).
\eeq
Here 
\[D {\boldsymbol u}=  \begin{pmatrix}
         \nabla u^1 \\
         \vdots \\
         \nabla u^N
        \end{pmatrix},\] 
        where 
\[\nabla u^\ell = \left(\frac{\partial u^\ell}{\partial x_1},\ldots,\frac{\partial u^\ell}{\partial x_n}\right)\,\, \text{for } \ell=1,\ldots,N,\]
and
\[\displaystyle{|D {\boldsymbol u}|=\sqrt{\sum_{\ell=1}^N\sum_{j=1}^n \left(\frac{\partial u^\ell}{\partial x_j}\right)^2}}.\] 
We also use the notation $\displaystyle u_j^\ell=\dfrac{\partial u^\ell}{\partial x_j}$. \\
We point out that along this paper the symbol $\,:\,$ stands for the scalar product of the matrices rows, namely 
$$
\mathcal{M}:\mathcal{N}=\sum_{i=1}^{q} \mathcal{M}^i \cdot \mathcal{N}^i=\sum_{i=1}^{q} \sum_{j=1}^{r} \mathcal{M}^i_j \mathcal{N}^i_j,
$$
where $\mathcal{M},\mathcal{N}$ are $(q \times r)$ real matrices, whereas $ \cdot$ denotes the scalar product of two real vectors. \\In the following we will denote the space of the $q\times r$ matrices as $\mathbb{R}^{q\times r}$ with $q,r \in \mathbb{N}$.
Observe that \eqref{weak} can be rewritten as
$$
\sum_{\ell=1}^N \int_{\Omega} |D {\boldsymbol u}|^{p-2}\langle \nabla u^\ell,  \nabla \varphi^\ell \rangle \, dx = \sum_{\ell=1}^N \int_{\Omega} f^\ell\varphi^\ell\, dx, \quad \forall {\bf\varphi} \in W^{1,p}_{0}(\Omega).
$$
In order to simplify the computations below, we define the vector
\begin{equation}\label{eq:vect}
D^{1,2}\boldsymbol{u}:=\left(\begin{array}{c}D\boldsymbol{u}:D\boldsymbol{u}_{1}\\
\vdots \\D\boldsymbol{u}:D\boldsymbol{u}_{i}\\\vdots
\\
D\boldsymbol{u}:D\boldsymbol{u}_{n}
\end{array}\right).
\end{equation}
Moreover we define
\[\|D^2 {\boldsymbol u} \|=\sqrt{\sum_{i=1}^N\sum_{j,k}(u^i_{jk})^2}.\]
In what follows we often use the inequality
\begin{equation}\label{eq:FraCo}
|D^{1,2} \boldsymbol{u}|\leq |D\boldsymbol{u}|\|D^2\boldsymbol{u}\|.
\end{equation}
\noindent We are now ready to prove Theorem \ref{teosecond}.
\begin{proof}[Proof of Theorem \ref{teosecond}]
Let $x_0\in \Omega$ and  $B=B_{2\rho(x_0)}\subset\subset\Omega$. Since $\boldsymbol u\in  C^1(\bar{\Omega})$, for any $\varepsilon>0$, let $\boldsymbol u_\varepsilon\in \boldsymbol u+W_0^{1,p}(B)$ be the unique solution to
\beq\label{system}\tag{$p$-$S_\varepsilon$}
\begin{cases}
-\operatorname{\bf div}((\varepsilon+|D{\boldsymbol u_\varepsilon}|^2)^\frac{p-2}{2}D{\boldsymbol u_\varepsilon})=   {\boldsymbol f}(x)& \mbox{in $B_{2\rho(x_0)}$}\\
{\boldsymbol u_\varepsilon}= \boldsymbol u  & \mbox{on  $\partial B_{2\rho(x_0)}$}, 
\end{cases}
\eeq
Testing \eqref{system} by  $\boldsymbol\varphi\in C_c^\infty(B)$ we get
\begin{eqnarray}\label{eq:log(-epsilon)}
\int_{B}(\varepsilon +|D{\boldsymbol u_\varepsilon}|^2)^{\frac{p-2}{2}}(D{\boldsymbol u_{\varepsilon}}: D\boldsymbol\varphi)dx
=\int_B\langle \boldsymbol{f},\boldsymbol\varphi\rangle dx.
\end{eqnarray}

The existence of such a weak solution follows by classical minimization procedure. Notice that by standard regularity results \cite{DiKaSc,  GT} the solution $\boldsymbol u_\varepsilon$ belongs to $C^{2,\alpha}(B)$. 
Fixing $i=1,\ldots,n$, using $\boldsymbol \varphi_i\in C_c^\infty(B)$ in \eqref{eq:log(-epsilon)}, we obtain
\begin{eqnarray*}
\int_B(\varepsilon +|D{\boldsymbol u_\varepsilon}|^2)^{\frac{p-2}{2}}(D{\boldsymbol u_{\varepsilon}}: D\boldsymbol\varphi_i)dx
=\int_B\langle \boldsymbol{f},\boldsymbol\varphi_i\rangle dx
\end{eqnarray*}
Integrating by parts, we get
\begin{eqnarray*}
&&\int_B(\varepsilon +|D{\boldsymbol u_\varepsilon}|^2)^{\frac{p-2}{2}}(D{\boldsymbol u_{\varepsilon,i}}: D\boldsymbol\varphi)\, dx\\\nonumber
&&+(p-2)\int_B(\varepsilon +|D{\boldsymbol u_\varepsilon}|^2)^{\frac{p-4}{2}}(D{\boldsymbol u_\varepsilon}: D{\boldsymbol u_{\varepsilon,i}})(D{\boldsymbol u_\varepsilon}: D\boldsymbol\varphi)\, dx
\\\nonumber
&=&\int_B\langle \boldsymbol{f}_i,\boldsymbol\varphi\rangle dx.
\end{eqnarray*}
For any $m\in\mathbb{R}^+$, let $G_{\frac 1m}(t)$ be defined by
\begin{equation*}				
 G_{\frac 1m}(t):=\left\{
\begin{array}{ll}
	0 & \hbox{ if } 0\leq t\leq \frac{1}{m},\\
	\displaystyle 2t- \frac{2}{m}  & \hbox{ if }  \frac{1}{m}< t< \frac{2}{m}, \\
	t & \hbox{ if } t \geq \frac{2}{m}. 
\end{array}
\right.
\end{equation*}
and for a given $\gamma>0$, 
\begin{equation}\label{eq:FraCo4}
 H_{\frac 1m}(t):= \frac{G_{\frac 1m}(t)}{t^{\gamma+1}}.
\end{equation} 
Let $\psi\in C_c^\infty(B_{2\rho}(x_0))$ such that $\psi=1$ in $B_\rho(x_0)$, $\psi=0$ in $(B_{2\rho}(x_0))^c$ and $|\nabla \psi|\leq\frac{2}{\rho}$ in the annular $B_{2\rho}(x_0)\setminus B_\rho(x_0)$. \\ \\
Let us define,
\begin{equation*}
\boldsymbol\varphi=\frac{H_{\frac 1m}(|x-y|)\psi^2}{(\varepsilon +|D\boldsymbol{u}_{\varepsilon}|^{2})^{\frac{\alpha}{2}}}\boldsymbol{u}_{\varepsilon,i}.
\end{equation*}\ \\
Then using $\varphi$ as a test function we deduce
\begin{eqnarray}\label{eq:princ}\nonumber
&&\int_B\frac{(\varepsilon +|D\boldsymbol{u}_{\varepsilon}|^2)^{\frac{p-2}{2}}}{(\varepsilon +|D\boldsymbol{u}_{\varepsilon}|^{2})^{\frac{\alpha}{2}}} |D\boldsymbol{u}_{\varepsilon,i}|^2H_{\frac 1m}(|x-y|)\psi^2 \, dx\\\nonumber
&&-\alpha \int_B\frac{(\varepsilon +|D\boldsymbol{u}_{\varepsilon}|^2)^{\frac{p-2}{2}}}{(\varepsilon +|D\boldsymbol{u}_{\varepsilon}|^2)^{\frac{\alpha+2}{2}}}(D\boldsymbol{u}_{\varepsilon,i}:\boldsymbol{u}_{\varepsilon,i}\left(D^{1,2}\boldsymbol{u}_{\varepsilon}\right)^T)H_{\frac 1m}(|x-y|)\psi^2\, dx\\
&&+(p-2)\int_B \frac{(\varepsilon +|D\boldsymbol{u}_{\varepsilon}|^2)^{\frac{p-4}{2}}}{(\varepsilon +|D\boldsymbol{u}_{\varepsilon}|^{2})^{\frac{\alpha}{2}}}(D\boldsymbol{u}_{\varepsilon}:D\boldsymbol{u}_{\varepsilon,i})^2H_{\frac 1m}(|x-y|)\psi^2\, dx\\\nonumber
&&-\alpha(p-2) \int_B\frac{(\varepsilon +|D\boldsymbol{u}_{\varepsilon}|^2)^{\frac{p-4}{2}}}{(\varepsilon +|D\boldsymbol{u}_{\varepsilon}|^2)^{\frac{\alpha+2}{2}}}(D\boldsymbol{u}_{\varepsilon}:D\boldsymbol{u}_{\varepsilon,i})(D\boldsymbol{u}_{\varepsilon}:\boldsymbol{u}_{\varepsilon,i}\left(D^{1,2}\boldsymbol{u}_{\varepsilon}\right)^T)H_{\frac 1m}(|x-y|)\psi^2\, dx\\\nonumber
\end{eqnarray}
\begin{eqnarray}
&=&\nonumber-2\int_B\frac{(\varepsilon +|D\boldsymbol{u}_{\varepsilon}|^2)^{\frac{p-2}{2}}}{(\varepsilon +|D\boldsymbol{u}_{\varepsilon}|^{2})^{\frac{\alpha}{2}}}(D\boldsymbol{u}_{\varepsilon,i}: \boldsymbol{u}_{\varepsilon,i}\nabla \psi^T)H_{\frac 1m}(|x-y|) \psi \  dx\\\nonumber
&&-\int_B\frac{(\varepsilon +|D\boldsymbol{u}_{\varepsilon}|^2)^{\frac{p-2}{2}}}{(\varepsilon +|D\boldsymbol{u}_{\varepsilon}|^{2})^{\frac{\alpha}{2}}}\left(D\boldsymbol{u}_{\varepsilon,i}: \boldsymbol{u}_{\varepsilon,i}\left(\nabla H_{\frac 1m}(|x-y|)\right)^T\right)\psi ^2 \, dx\\
\nonumber
&&-2(p-2)\int_B\frac{(\varepsilon +|D\boldsymbol{u}_{\varepsilon}|^2)^{\frac{p-4}{2}}}{(\varepsilon +|D\boldsymbol{u}_{\varepsilon}|^{2})^{\frac{\alpha}{2}}}(D\boldsymbol{u}_{\varepsilon}:D\boldsymbol{u}_{\varepsilon,i})(D\boldsymbol{u}_{\varepsilon}: \boldsymbol{u}_{\varepsilon,i}\nabla \psi^T)H_{\frac 1m}(|x-y|)\psi \, dx\\\nonumber
&&-(p-2)\int_B\frac{(\varepsilon +|D\boldsymbol{u}_{\varepsilon}|^2)^{\frac{p-4}{2}}}{(\varepsilon +|D\boldsymbol{u}_{\varepsilon}|^{2})^{\frac{\alpha}{2}}}(D\boldsymbol{u}_{\varepsilon}:D\boldsymbol{u}_{\varepsilon,i})\left(D\boldsymbol{u}_{\varepsilon}: \boldsymbol{u}_{\varepsilon,i}\left(\nabla H_{\frac 1m}(|x-y|)\right)^T\right)\psi^2 \, dx\\
\nonumber
&&+\int_B\langle \boldsymbol{f}_i,\boldsymbol\varphi\rangle dx.
\end{eqnarray}
Exploiting also a nice computation found  in \cite{M}, we deduce that
\begin{eqnarray}\label{terminesinistra}
&&\sum_{i=1}^{n}\left\{ \int_B\frac{(\varepsilon +|D\boldsymbol{u}_{\varepsilon}|^2)^{\frac{p-2}{2}}}{(\varepsilon +|D\boldsymbol{u}_{\varepsilon}|^{2})^{\frac{\alpha}{2}}} |D\boldsymbol{u}_{\varepsilon,i}|^2H_{\frac 1m}(|x-y|)\psi^2 \, dx\right.
\\\nonumber
&&-\alpha \int_B\frac{(\varepsilon +|D\boldsymbol{u}_{\varepsilon}|^2)^{\frac{p-2}{2}}}{(\varepsilon +|D\boldsymbol{u}_{\varepsilon}|^2)^{\frac{\alpha+2}{2}}}(D\boldsymbol{u}_{\varepsilon,i}:\boldsymbol{u}_{\varepsilon,i}\left(D^{1,2}\boldsymbol{u}_{\varepsilon}\right)^T)H_{\frac 1m}(|x-y|)\psi^2\, dx\\\nonumber
&&+(p-2)\int_B \frac{(\varepsilon +|D\boldsymbol{u}_{\varepsilon}|^2)^{\frac{p-4}{2}}}{(\varepsilon +|D\boldsymbol{u}_{\varepsilon}|^{2})^{\frac{\alpha}{2}}}(D\boldsymbol{u}_{\varepsilon}:D\boldsymbol{u}_{\varepsilon,i})^2H_{\frac 1m}(|x-y|)\psi^2\, dx\\\nonumber
\\\nonumber
&&\left.-\alpha(p-2) \int_B\frac{(\varepsilon +|D\boldsymbol{u}_{\varepsilon}|^2)^{\frac{p-4}{2}}}{(\varepsilon +|D\boldsymbol{u}_{\varepsilon}|^2)^{\frac{\alpha+2}{2}}}(D\boldsymbol{u}_{\varepsilon}:D\boldsymbol{u}_{\varepsilon,i})(D\boldsymbol{u}_{\varepsilon}:\boldsymbol{u}_{\varepsilon,i}D^{1,2}\boldsymbol{u}_{\varepsilon}^T)H_{\frac 1m}(|x-y|)\psi^2\, dx\right\}\\\nonumber
&&=\int_B(\varepsilon +|D\boldsymbol{u}_{\varepsilon}|^2)^{\frac{p-2-\alpha}{2}} \|D^2\boldsymbol{u}_{\varepsilon}\|^2H_{\frac 1m}(|x-y|)\psi^2 \, dx\\\nonumber
&&+(p-2-\alpha)\int_B(\varepsilon +|D\boldsymbol{u}_{\varepsilon}|^2)^{\frac{p-4-\alpha}{2}} \sum_{i=1}^n(D\boldsymbol{u}_\varepsilon:D\boldsymbol{u}_{\varepsilon,i})^2H_{\frac 1m}(|x-y|)\psi^2 \, dx\\\nonumber
&&-\alpha(p-2)\int_B(\varepsilon +|D\boldsymbol{u}_{\varepsilon}|^2)^{\frac{p-6-\alpha}{2}}\sum_{k=1}^N\langle D^{1,2}\boldsymbol{u}_{\varepsilon}, \nabla u^k\rangle \langle D^{1,2}\boldsymbol{u}_{\varepsilon}, \nabla u^k\rangle H_{\frac 1m}(|x-y|)\psi^2 \, dx. 
\end{eqnarray}
We distinguish two case: If $1<p\leq 2$ from \eqref{terminesinistra} we deduce 
\begin{eqnarray}\label{FraCo1}
&&\sum_{i=1}^{n}\int_B{(\varepsilon +|D\boldsymbol{u}_{\varepsilon}|^2)^{\frac{p-2-\alpha}{2}}} |D\boldsymbol{u}_{\varepsilon,i}|^2H_{\frac 1m}(|x-y|)\psi^2 \, dx \\\nonumber
&&\geq\int_B(\varepsilon +|D\boldsymbol{u}_{\varepsilon}|^2)^{\frac{p-2-\alpha}{2}} \|D^2\boldsymbol{u}_{\varepsilon}\|^2H_{\frac 1m}(|x-y|)\psi^2 \, dx\\\nonumber
&&+(p-2-\alpha)\int_B(\varepsilon +|D\boldsymbol{u}_{\varepsilon}|^2)^{\frac{p-4-\alpha}{2}} \sum_{i=1}^n(D\boldsymbol{u}_\varepsilon:D\boldsymbol{u}_{\varepsilon,i})^2H_{\frac 1m}(|x-y|)\psi^2 \, dx\\\nonumber
&&\geq (p-1-\alpha)\int_B(\varepsilon +|D\boldsymbol{u}_{\varepsilon}|^2)^{\frac{p-2-\alpha}{2}} \|D^2\boldsymbol{u}_{\varepsilon}\|^2H_{\frac 1m}(|x-y|)\psi^2 \, dx,
\end{eqnarray}
where in the last line we used Cauchy-Schwarz inequality. In the case $p>2$, by \eqref{eq:FraCo}, we observe that 
\begin{eqnarray}\label{eq:FraCo2}
&&\sum_{k=1}^N\langle D^{1,2}\boldsymbol{u}_{\varepsilon}, \nabla u^k\rangle \langle D^{1,2}\boldsymbol{u}_{\varepsilon}, \nabla u^k\rangle=\sum_{k=1}^N\langle D^{1,2}\boldsymbol{u}_{\varepsilon}, \nabla u^k\rangle^2\\\nonumber
&&=\|D\boldsymbol{u}_{\varepsilon}D^{1,2}\boldsymbol{u}_{\varepsilon}\|^2\leq \|D\boldsymbol{u}_{\varepsilon}\|^4\|D^{1,2}\boldsymbol{u}_{\varepsilon}\|^2.
\end{eqnarray}
Therefore plugging \eqref{eq:FraCo2} in \eqref{terminesinistra} we obtain 
\begin{eqnarray}\label{FraCo11}
&&\sum_{i=1}^{n}\int_B{(\varepsilon +|D\boldsymbol{u}_{\varepsilon}|^2)^{\frac{p-2-\alpha}{2}}} |D\boldsymbol{u}_{\varepsilon,i}|^2H_{\frac 1m}(|x-y|)\psi^2 \, dx \\\nonumber
&\geq&\min\{1,(p-1)(1-\alpha)\}\int_B(\varepsilon +|D\boldsymbol{u}_{\varepsilon}|^2)^{\frac{p-2-\alpha}{2}} \|D^2\boldsymbol{u}_{\varepsilon}\|^2H_{\frac 1m}(|x-y|)\psi^2 \, dx.
\end{eqnarray}

Let us continue by estimating each single term on the right of \eqref{eq:princ}, at first considering their norm and then the sum over $i$.  Observe that 
\begin{eqnarray}\label{aa}
\nonumber &&|(D\boldsymbol{u}_{\varepsilon,i}: \boldsymbol{u}_{\varepsilon,i}\nabla \psi^T)|\leq |D\boldsymbol{u}_{\varepsilon,i}| |\boldsymbol{u}_{\varepsilon,i}\nabla \psi^T|=|D\boldsymbol{u}_{\varepsilon,i}|\sqrt{\sum_{j=1}^N\sum_{k=1}^n(u^j_{\varepsilon,i}\psi_k)^2}\\&&=|D\boldsymbol{u}_{\varepsilon,i}|\sqrt{|\nabla \psi|^2\sum_{j=1}^N(u^j_{\varepsilon,i})^2}\leq |D\boldsymbol{u}_{\varepsilon,i}| |D\boldsymbol{u}_{\varepsilon}||\nabla \psi|.
\end{eqnarray}
Using weighted Young inequality and \eqref{aa}, taking into account that $G_{\frac 1m}(|x-y|)\leq|x-y|$, we have
\begin{eqnarray}\label{eq:13est1}\nonumber&&2\int_B(\varepsilon +|D\boldsymbol{u}_{\varepsilon}|^2)^{\frac{p-2-\alpha}{2}}|(D\boldsymbol{u}_{\varepsilon,i}: \boldsymbol{u}_{\varepsilon,i}\nabla \psi^T)|H_{\frac 1m }(|x-y|)\psi \, dx\\\nonumber
&\leq& 2\int_B(\varepsilon +|D\boldsymbol{u}_{\varepsilon}|^2)^{\frac{p-2-\alpha}{2}}|D\boldsymbol{u}_{\varepsilon,i}||D\boldsymbol{u}_{\varepsilon}||\nabla \psi|H_{\frac 1m }(|x-y|)\psi \, dx\\\nonumber
&\leq& 2\delta\int_B(\varepsilon +|D\boldsymbol{u}_{\varepsilon}|^2)^{\frac{p-2-\alpha}{2}}|D\boldsymbol{u}_{\varepsilon,i}|^2\frac {\psi^2}{|x-y|^{\gamma}} \, dx\\\nonumber
&&+\frac{1}{2\delta}\int_B(\varepsilon +|D\boldsymbol{u}_{\varepsilon}|^2)^{\frac{p-2-\alpha}{2}}|D\boldsymbol{u}_{\varepsilon}|^2|\nabla \psi|^2\frac {1}{|x-y|^{\gamma}} \, dx\\
&\leq& 2\delta\int_B(\varepsilon +|D\boldsymbol{u}_{\varepsilon}|^2)^{\frac{p-2-\alpha}{2}}|D\boldsymbol{u}_{\varepsilon,i}|^2\frac {\psi^2}{|x-y|^{\gamma}}  \, dx\\\nonumber
&&+\frac{1}{2\delta}\int_B{(\varepsilon +|D\boldsymbol{u}_{\varepsilon}|^2)^{\frac{p-\alpha}{2}}}|\nabla \psi|^2\frac {1}{|x-y|^{\gamma}} \, dx
\\\nonumber
&\leq& 2\delta\int_B(\varepsilon +|D\boldsymbol{u}_{\varepsilon}|^2)^{\frac{p-2-\alpha}{2}}|D\boldsymbol{u}_{\varepsilon,i}|^2\frac {\psi^2}{|x-y|^{\gamma}} \, dx+C(p,\alpha,\gamma,\delta,\rho),
\end{eqnarray}
where in the last line we used the fact   $p>\alpha$ and $|D\boldsymbol{u}_{\varepsilon}|\leq C$ for some constant uniformly on $\varepsilon$ (see \cite{DiKaSc}).

Summing up on $i=1,\dots,n$
\begin{eqnarray}\label{magg1}
 &&2\sum_{i=1}^{n}\int_B(\varepsilon +|D\boldsymbol{u}_{\varepsilon}|^2)^{\frac{p-2-\alpha}{2}}|(D\boldsymbol{u}_{\varepsilon,i}: \boldsymbol{u}_{\varepsilon}\nabla \psi^T)|H_{\frac 1m }(|x-y|)\psi \, dx
 \\\nonumber
 &\leq& 2\delta\int_B(\varepsilon +|D\boldsymbol{u}_{\varepsilon}|^2)^{\frac{p-2-\alpha}{2}}\|D^2\boldsymbol{u}_{\varepsilon}\|^2\frac {\psi^2}{|x-y|^{\gamma}} \, dx+C(n,p,\alpha,\gamma,\delta, \rho),
\end{eqnarray}

As in \eqref{eq:13est1},  by weighted Young inequality,  we get
\begin{eqnarray}\label{eq:13est2}\nonumber
&&2|(p-2)|\int_B(\varepsilon +|D\boldsymbol{u}_{\varepsilon}|^2)^{\frac{p-4-\alpha}{2}}|(D\boldsymbol{u}_{\varepsilon}:D\boldsymbol{u}_{\varepsilon,i})||(D\boldsymbol{u}_{\varepsilon}: \boldsymbol{u}_{\varepsilon,i}\nabla \psi^T)|H_{\frac 1m}(|x-y|)\psi \, dx\\\nonumber
&\leq& 2|p-2|\int_B(\varepsilon +|D\boldsymbol{u}_{\varepsilon}|^2)^{\frac{p-2-\alpha}{2}}|D\boldsymbol{u}_{\varepsilon}||D\boldsymbol{u}_{\varepsilon,i}||\nabla\psi| H_{\frac 1m }(|x-y|)\psi \, dx\\\nonumber
&\leq& 2\delta|p-2|\int_B(\varepsilon +|D\boldsymbol{u}_{\varepsilon}|^2)^{\frac{p-2-\alpha}{2}}|D\boldsymbol{u}_{\varepsilon,i}|^2\frac {\psi^2}{|x-y|^{\gamma}} \,dx\\&&+\frac{|p-2|}{2\delta}\int_B(\varepsilon +|D\boldsymbol{u}_{\varepsilon}|^2)^{\frac{p-2-\alpha}{2}}|D\boldsymbol{u}_{\varepsilon}|^2|\nabla \psi|^2\frac{1}{|x-y|^{\gamma}} \, dx
\\\nonumber
&\leq& 2\delta|p-2|\int_B(\varepsilon +|D\boldsymbol{u}_{\varepsilon}|^2)^{\frac{p-2-\alpha}{2}}|D\boldsymbol{u}_{\varepsilon,i}|^2\frac {\psi^2}{|x-y|^{\gamma}} \,dx\\\nonumber&&+\frac{|p-2|}{2\delta}\int_B{(\varepsilon +|D\boldsymbol{u}_{\varepsilon}|^2)^{\frac{p-\alpha}{2}}}|\nabla \psi|^2\frac{1}{|x-y|^{\gamma}} \, dx\\\nonumber
&\leq& 2\delta|p-2|\int_B(\varepsilon +|D\boldsymbol{u}_{\varepsilon}|^2)^{\frac{p-2-\alpha}{2}}|D\boldsymbol{u}_{\varepsilon,i}|^2\frac {\psi^2}{|x-y|^{\gamma}} \,dx+C(p,\alpha,\gamma,\delta, \rho).
\end{eqnarray}
Then
\begin{eqnarray}\label{magg2}
	&&\nonumber2|p-2|\sum_{i=1}^{n}\int_B(\varepsilon +|D\boldsymbol{u}_{\varepsilon}|^2)^{\frac{p-4-\alpha}{2}}|(D\boldsymbol{u}_{\varepsilon}:D\boldsymbol{u}_{\varepsilon,i})||(D\boldsymbol{u}_{\varepsilon}: \boldsymbol{u}_{\varepsilon,i}\nabla \psi^T)|H_{\frac 1m}(|x-y|)\psi \, dx
	\\
	&\leq& 2\delta|p-2|\int_B(\varepsilon +|D\boldsymbol{u}_{\varepsilon}|^2)^{\frac{p-2-\alpha}{2}}\|D^2\boldsymbol{u}_{\varepsilon}\|^2\frac {\psi^2}{|x-y|^{\gamma}} \, dx\\\nonumber && +C(n,p,\alpha,\gamma,\delta, \rho).
\end{eqnarray}

Arguing as in \eqref{eq:13est1}, following \eqref{aa},  inasmuch
\begin{equation}\label{bb}
\left|\left(D\boldsymbol{u}_{\varepsilon,i}: \boldsymbol{u}_{\varepsilon ,i}\left(\nabla H_{\frac 1m}(|x-y|)\right)^T\right)\right|\leq |D\boldsymbol{u}_{\varepsilon,i}||D\boldsymbol{u}_{\varepsilon}||\nabla H_{\frac 1m}(|x-y|)|.
\end{equation}
By \eqref{eq:FraCo4}, \eqref{bb} and weighted Young inequality,  we get
\begin{eqnarray}\label{11}    
 \nonumber &&\int_B(\varepsilon +|D\boldsymbol{u}_{\varepsilon}|^2)^{\frac{p-2-\alpha}{2}}\left|\left(D\boldsymbol{u}_{\varepsilon,i}: \boldsymbol{u}_{\varepsilon,i}\left(\nabla H_{\frac 1m}(|x-y|)\right)^T\right)\right|\psi ^2 \, dx \\  \nonumber&\leq&\int_B(\varepsilon +|D\boldsymbol{u}_{\varepsilon}|^2)^{\frac{p-2-\alpha}{2}}|D\boldsymbol{u}_{\varepsilon,i}||D\boldsymbol{u}_{\varepsilon}||\nabla H_{\frac 1m}(|x-y|)|\psi ^2 \,dx\\ 
\nonumber&\leq&\int_B(\varepsilon +|D\boldsymbol{u}_{\varepsilon}|^2)^{\frac{p-2-\alpha}{2}}|D\boldsymbol{u}_{\varepsilon,i}||D\boldsymbol{u}_{\varepsilon}|\frac{G'_{\frac 1m}(|x-y|)+(\gamma+1)\frac{G_m(|x-y|)}{|x-y|}}{|x-y|^{\gamma+1}}\psi ^2 \,dx\\\nonumber
 &\leq&C_\gamma\delta\int_B(\varepsilon +|D\boldsymbol{u}_{\varepsilon}|^2)^{\frac{p-2-\alpha}{2}}|D\boldsymbol{u}_{\varepsilon,i}|^2\frac{\psi ^2}{|x-y|^\gamma} \,dx \\ 
 \nonumber &&+\frac{C_\gamma}{4\delta} \int_B(\varepsilon +|D\boldsymbol{u}_{\varepsilon}|^2)^{\frac{p-\alpha}{2}}\frac {\psi^2}{|x-y|^{\gamma+2}}\,dx \\  
  \nonumber&\leq& C_\gamma \delta\int_B(\varepsilon +|D\boldsymbol{u}_{\varepsilon}|^2)^{\frac{p-2-\alpha}{2}}|D\boldsymbol{u}_{\varepsilon,i}|^2\frac {\psi^2}{|x-y|^{\gamma}}\,dx + C(p,\alpha,\gamma,\delta, \rho),
\end{eqnarray}
where we have used the fact that $\gamma+2<n$, and $C_{\gamma}$ and $C(p,\alpha,\gamma,\delta, \rho)$ are positive constants.
This implies that
\begin{eqnarray}\label{magg3}
	&&\sum_{i=1}^{n}\int_B(\varepsilon +|D\boldsymbol{u}_{\varepsilon}|^2)^{\frac{p-2-\alpha}{2}}\left|\left(D\boldsymbol{u}_{\varepsilon,i}: \boldsymbol{u}_{\varepsilon,i}\left(\nabla H_{\frac 1m}(|x-y|)\right)^T\right)\right|\psi ^2 \, dx
	\\\nonumber
	&\leq&C_\gamma \delta\int_B(\varepsilon +|D\boldsymbol{u}_{\varepsilon}|^2)^{\frac{p-2-\alpha}{2}}\|D^2\boldsymbol{u}_{\varepsilon}\|^2\frac {\psi^2}{|x-y|^{\gamma}}\,dx + C(n,p,\alpha,\gamma,\delta).
\end{eqnarray}
Similarly,  we see that
\begin{eqnarray}\label{22}
 \nonumber &&|p-2|\int_B(\varepsilon +|D\boldsymbol{u}_{\varepsilon}|^2)^{\frac{p-4-\alpha}{2}}\left|(D\boldsymbol{u}_{\varepsilon}:D\boldsymbol{u}_{\varepsilon,i})\right|\left|\left(D\boldsymbol{u}_{\varepsilon}: \boldsymbol{u}_{\varepsilon,i}\left(\nabla H_{\frac 1m }(|x-y|)\right)^T\right)\right|\psi^2 \, dx\\ 
 \nonumber &\leq& |p-2|\int_B{(\varepsilon +|D\boldsymbol{u}_{\varepsilon}|^2)^{\frac{p-4-\alpha}{2}}}(\varepsilon +|D\boldsymbol{u}_{\varepsilon}|^2)|D\boldsymbol{u}_{\varepsilon,i}|
 |D\boldsymbol{u}_{\varepsilon}| |\nabla H_{\frac 1m}(|x-y|)|\psi^2 \, dx \\ 
 \nonumber &\leq& C_{\gamma}\delta|p-2|\int_B(\varepsilon +|D\boldsymbol{u}_{\varepsilon}|^2)^{\frac{p-2-\alpha}{2}}|D\boldsymbol{u}_{\varepsilon,i}|^2
 \frac{\psi^2}{|x-y|^{\gamma}}\, dx \\ 
 &&+C_{\gamma}\frac{|p-2|}{4\delta }\int_B(\varepsilon +|D\boldsymbol{u}_{\varepsilon}|^2)^{\frac{p-\alpha}{2}} \frac{\psi^2}{|x-y|^{\gamma+2}} \, dx \\ 
 \nonumber &\leq&C_{\gamma}\delta|p-2|\int_B(\varepsilon +|D\boldsymbol{u}_{\varepsilon}|^2)^{\frac{p-2-\alpha}{2}}|D\boldsymbol{u}_{\varepsilon,i}|^2
 \frac{\psi^2}{|x-y|^{\gamma}} \, dx +C(p,\alpha,\gamma,\delta,\rho).
\end{eqnarray}
So we conclude that
\begin{eqnarray}\label{magg4}
	&&\nonumber|p-2|\sum_{i=1}^{n}\int_B(\varepsilon +|D\boldsymbol{u}_{\varepsilon}|^2)^{\frac{p-4-\alpha}{2}}\left|(D\boldsymbol{u}_{\varepsilon}:D\boldsymbol{u}_{\varepsilon,i})\right|\left|\left(D\boldsymbol{u}_{\varepsilon}: \boldsymbol{u}_{\varepsilon,i}\left(\nabla H_{\frac 1m }(|x-y|)\right)^T\right)\right|\psi^2 \, dx\\ 
	&\leq& C_\gamma\delta|p-2|\int_B(\varepsilon +|D\boldsymbol{u}_{\varepsilon}|^2)^{\frac{p-2-\alpha}{2}}\|D^2\boldsymbol{u}_{\varepsilon}\|^2\frac {\psi^2}{|x-y|^{\gamma}} \, dx+C(n,p,\alpha,\gamma,\delta, \rho).
\end{eqnarray}
Using estimations \eqref{FraCo1}, \eqref{FraCo11}, \eqref{magg1}, \eqref{magg2}, \eqref{magg3}, \eqref{magg4}, exploiting the Fatou Lemma,  for $m\to+\infty$ and redenoting by $C_\gamma:=2+C_\gamma$ we deduce:

if $1<p\leq 2$
\begin{eqnarray}\label{eq:princ11}
&&(p-1-\alpha)\int_B(\varepsilon +|D\boldsymbol{u}_{\varepsilon}|^2)^{\frac{p-2-\alpha}{2}}\|D^2\boldsymbol{u}_{\varepsilon}\|^2\frac{\psi^2}{|x-y|^\gamma}\, dx\\\nonumber
&\leq& C_\gamma\delta \max\{(p-1), (3-p)\}\int_B(\varepsilon +|D\boldsymbol{u}_{\varepsilon}|^2)^{\frac{p-2-\alpha}{2}}\|D^2\boldsymbol{u}_{\varepsilon}\|^2\frac{\psi^2}{|x-y|^\gamma} \, dx
\\\nonumber
\nonumber&&
+\sum_{i=1}^{n}\int_B\langle \boldsymbol{f}_i,\boldsymbol\varphi\rangle dx
+C(n, p,\alpha,\gamma, \delta,\rho);
\end{eqnarray}

if $p> 2$
\begin{eqnarray}\label{eq:princ111}
&&\min \{1,(p-1)(1-\alpha)\}\int_B(\varepsilon +|D\boldsymbol{u}_{\varepsilon}|^2)^{\frac{p-2-\alpha}{2}}\|D^2\boldsymbol{u}_{\varepsilon}\|^2\frac{\psi^2}{|x-y|^\gamma}\, dx\\\nonumber
&\leq& C_\gamma\delta \max\{(p-1), (3-p)\}\int_B(\varepsilon +|D\boldsymbol{u}_{\varepsilon}|^2)^{\frac{p-2-\alpha}{2}}\|D^2\boldsymbol{u}_{\varepsilon}\|^2\frac{\psi^2}{|x-y|^\gamma} \, dx
\\\nonumber
\nonumber&&
+\sum_{i=1}^{n}\int_B\langle \boldsymbol{f}_i,\boldsymbol\varphi\rangle dx
+C(n, p,\alpha,\gamma, \delta,\rho).
\end{eqnarray}

There exists $\delta=\delta(p,\alpha,\gamma)$ such that  both 
\[(p-1-\alpha) - C_\gamma\delta \max\{(p-1), (3-p)\}:= \beta>0,\]
and
\[\min \{1,(p-1)(1-\alpha)\} - C_\gamma\delta \max\{(p-1), (3-p)\}:= \beta>0,\]
hold.
Therefore from \eqref{eq:princ11} and \eqref{eq:princ111} we obtain
\begin{eqnarray}\label{eq:princ1i}
\\\nonumber
&&\beta\int_B(\varepsilon +|D\boldsymbol{u}_{\varepsilon}|^2)^{\frac{p-2-\alpha}{2}}\|D^2\boldsymbol{u}_{\varepsilon}\|^2\frac{\psi^2}{|x-y|^\gamma} \, dx\\\nonumber
&&\leq\int_B \sum_{i=1}^{n}\sum_{j=1}^N {f}^j_{i}(x)\varphi^j dx
+C(n,p,\alpha,\gamma, \delta,\rho)\leq C,
\end{eqnarray}
where in the last inequality we have used $\boldsymbol f\in W_{loc}^{1,\frac{n}{n-\gamma-s}}(\Omega)$ and where $C=C(\boldsymbol f, n,p,\alpha,\gamma,\rho)$ is a positive constant. \\
Let us consider a compact set $K\subset \subset B \setminus Z_{\boldsymbol u}$. We note that, by \cite{DiKaSc}, $\boldsymbol{u}_\varepsilon$ is uniformly bounded in $C^{1,\beta}(K)$, for some $0<\beta<1$. Moreover, by Dirichlet datum in \eqref{system}, it follows that \begin{equation}\label{evvai}
\boldsymbol{u}_\varepsilon\rightarrow \boldsymbol u \quad \text{in the norm }\|\cdot\|_{C^{1,\beta}(K)}.
\end{equation}\  
By Schauder estimates (that it can be applied to each component) we remark that $$\|\boldsymbol{u}_\varepsilon\|_{C^{2,\beta}(K)}\leq C$$
where $\beta\in (0,1)$ and $C$ is a positive constant not depending on $\varepsilon$. So, by last inequality,  it follows that the second limit in \eqref{evvai} holds in $C^2(K).$ Therefore, exploiting Fatou's Lemma, we deduce that 
\begin{eqnarray*}%\label{eq:princ1iSs2}
\int_{B_\rho(x_0)\setminus Z_{\boldsymbol u}}  \frac{|D\boldsymbol{u}|^{p-2-\alpha} \|D^2\boldsymbol{u}\|^2}{|x-y|^\gamma} \, dx\leq C,
\end{eqnarray*}
where $C=C(\boldsymbol f, n,p,\alpha,\gamma,\rho)$ is a positive constant. \\
If $\tilde{\Omega}\subset\subset \Omega$,  by finite covering argument, 
\begin{eqnarray*}
\int_{\tilde{\Omega}\setminus Z_{\boldsymbol u}} \frac{|D\boldsymbol{u}|^{p-2-\alpha} \|D^2\boldsymbol{u}\|^2}{|x-y|^\gamma} \, dx\leq C.
\end{eqnarray*}
This conclude the proof.
\end{proof}
\noindent Now we are in the position to prove:
\begin{proof}[Proof of Theorem \ref{Reg2}] For any  $m\in\mathbb{R}^+$ let $G_{\frac 1m}(t)$ and $$\displaystyle H_{\frac 1m}(t):= \frac{G_{\frac 1m}(t)}{t^{\gamma+1}},$$ as in Theorem~\ref{teosecond}. Moreover let $x_0\in \Omega$ and $\rho>0$ such that $B=B_{2\rho}(x_0)\subset \Omega$. Define $\varphi\in C_c^\infty(\Omega)$ such that $\varphi=1$ on $B_{\rho}(x_0)$, $\displaystyle |\nabla \varphi|<\frac{2}{r}$ on $B_{2\rho}(x_0)\setminus B_{\rho}(x_0)$ and $\varphi=0$ otherwise. Let $\boldsymbol u_\varepsilon$  be the solution to \eqref{system} with $\boldsymbol f(x):=\boldsymbol f(x,\boldsymbol u(x))$ and $\boldsymbol u$ solution to \eqref{system1}.
\\  Consider the test function $\boldsymbol \psi=(0,\cdots,\psi^i,\cdots,0)$, where  
$$\psi^i=\frac{\varphi^2}{(\frac 1m +|D\boldsymbol u_\varepsilon|^2)^\frac{\sigma}{2}}H_{\frac 1m}(|x-y|).$$
By \eqref{system} and \eqref{eq:FraCo} we obtain 
\begin{eqnarray}\label{eq:stimapes}\\\nonumber
&&\int_{\Omega}f^i\frac{\varphi^2}{(\frac 1m +|D\boldsymbol u_\varepsilon|^2)^\frac{\sigma}{2}}H_{\frac 1m}(|x-y|)\, dx=\int_\Omega (\varepsilon +|D{\boldsymbol u_\varepsilon}|^2)^{\frac{p-2}{2}}\langle\nabla u^i_\varepsilon, \nabla \psi^i\rangle\, dx\\\nonumber
&=&-\sigma\int_\Omega(\varepsilon +|D{\boldsymbol u_\varepsilon}|^2)^{\frac{p-2}{2}}\langle \nabla u^i_\varepsilon, D^{1,2}\boldsymbol{u}_\varepsilon\rangle\frac{H_{\frac 1m}(|x-y|)\varphi^2}{(\frac 1m+|D\boldsymbol u_\varepsilon|^2)^{\frac{\sigma+2}{2}}}\, dx\\\nonumber
&&+
2\int_\Omega(\varepsilon +|D{\boldsymbol u_\varepsilon}|^2)^{\frac{p-2}{2}}\langle \nabla u^i_\varepsilon, \nabla \varphi\rangle\frac{H_{\frac 1m}(|x-y|)\varphi}{(\frac 1m+|D\boldsymbol u_\varepsilon|^2)^{\frac{\sigma}{2}}}\, dx \\\nonumber
&&+\int_\Omega (\varepsilon +|D{\boldsymbol u_\varepsilon}|^2)^{\frac{p-2}{2}}\langle\nabla u^i_\varepsilon, \nabla H_{\frac 1m}(|x-y|)\rangle\frac{\varphi^2}{(\frac 1m+|D\boldsymbol u_\varepsilon|^2)^{\frac{\sigma}{2}}}\, dx 
\\\nonumber
&\leq& \sigma\int_\Omega (\varepsilon +|D{\boldsymbol u_\varepsilon}|^2)^{\frac{p-1}{2}}|D{\boldsymbol u_\varepsilon}|\|D^2\boldsymbol{u}_\varepsilon\| \frac{H_{\frac 1m}(|x-y|)\varphi^2}{(\frac 1m+|D\boldsymbol u_\varepsilon|^2)^{\frac{\sigma+2}{2}}}\, dx
\\\nonumber
&&+
2\int_{B_{2\rho}}(\varepsilon +|D{\boldsymbol u_\varepsilon}|^2)^{\frac{p-1}{2}}|\nabla \varphi|\frac{H_{\frac 1m}(|x-y|)}{(\frac 1m+|D\boldsymbol u_\varepsilon|^2)^{\frac{\sigma}{2}}}\, dx
\\\nonumber
&&+\int_{B_{2\rho}} (\varepsilon +|D{\boldsymbol u_\varepsilon}|^2)^{\frac{p-1}{2}}|\nabla H_{\frac 1m}(|x-y|)|\frac{1}{(\frac 1m+|D\boldsymbol u_\varepsilon|^2)^{\frac{\sigma}{2}}}\, dx. 
\end{eqnarray}
We estimate the three terms in the right hand side of \eqref{eq:stimapes}. By weighted Young inequality we have
\begin{eqnarray}\label{eq:stimapes1}
&&\sigma\int_\Omega (\varepsilon +|D{\boldsymbol u_\varepsilon}|^2)^{\frac{p-1}{2}}|D{\boldsymbol u_\varepsilon}|\|D^2\boldsymbol{u}_\varepsilon\| \frac{H_{\frac 1m}(|x-y|)\varphi^2}{(\frac 1m+|D\boldsymbol u_\varepsilon|^2)^{\frac{\sigma+2}{2}}}\, dx\\\nonumber
&&\nonumber \leq \theta\sigma\int_\Omega\frac{H_{\frac 1m}(|x-y|)\varphi^2}{(\frac 1m+|D\boldsymbol u_\varepsilon|^2)^{\frac{\sigma}{2}}} \, dx+\frac{\sigma}{4\theta}\int_\Omega \frac{H_{\frac 1m}(|x-y|)\varphi^2}{(\frac 1m+|D\boldsymbol u_\varepsilon|^2)^{\frac{\sigma+2}{2}}}(\varepsilon+|D{\boldsymbol u_\varepsilon}|^2)^{(p-1)}\|D^2\boldsymbol{u}_\varepsilon\|^2\, dx. 
\end{eqnarray}	
For the second term, since $|D\boldsymbol{u}_{\varepsilon}|\leq C$ for some constant uniformly on $\varepsilon$ (see \cite{DiKaSc}), we deduce
\begin{eqnarray}\label{eq:stimapes2}
&&\int_{B_{2\rho}}(\varepsilon +|D{\boldsymbol u_\varepsilon}|^2)^{\frac{p-1}{2}}|\nabla \varphi|\frac{H_{\frac 1m}(|x-y|)}{(\frac 1m+|D\boldsymbol u_\varepsilon|^2)^{\frac{\sigma}{2}}}\, dx\\\nonumber
&\leq& C(\rho) \int_{B_{2\rho}} \left (\varepsilon+|D{\boldsymbol u_\varepsilon}|^2\right)^{\frac{p-1}{2}-\frac{\sigma}{2}}\frac{1}{|x-y|^\gamma}\, dx\leq C(p,\rho, \sigma, \gamma),
\end{eqnarray}  
where we choose $m=\frac {1}{\varepsilon}$
and we use that $\sigma <(p-1)$ and $\gamma< n-2$. We observe that the constant $C$ does not depend on the position of $y$ in $\Omega$.	
Similarly to above we obtain
\begin{eqnarray}\label{eq:stimapes3}
&&\int_{B_{2\rho}} (\varepsilon +|D{\boldsymbol u_\varepsilon}|^2)^{\frac{p-1}{2}}|\nabla H_{\frac 1m}(|x-y|)|\frac{1}{(\frac 1m+|D\boldsymbol u_\varepsilon|^2)^{\frac{\sigma}{2}}}\, dx\\\nonumber &\leq& C_\gamma\int_{B_{2\rho}}  \left (\varepsilon+|D{\boldsymbol u}_\varepsilon|^2\right)^{\frac{p-1}{2}-\frac{\sigma}{2}}\frac{1}{|x-y|^{\gamma+1}}\, dx \leq C(p,\rho, \sigma, \gamma).
\end{eqnarray}	
Using estimates \eqref{eq:stimapes1}, \eqref{eq:stimapes2} and \eqref{eq:stimapes3} in \eqref{eq:stimapes} we deduce,  recalling that,  $f^i\geq \tau >0$,
\begin{eqnarray}\nonumber\\\nonumber
&&\tau\int_{\Omega}\frac{H_{\frac 1m}(|x-y|)\varphi^2}{(\frac 1m +|D\boldsymbol u_\varepsilon|^2)^\frac{\sigma}{2}}\, dx\leq \theta\sigma\int_\Omega\frac{H_{\frac 1m}(|x-y|)\varphi^2}{(\frac 1m+|D\boldsymbol u_\varepsilon|^2)^{\frac{\sigma}{2}}} \, dx \\\nonumber &&+\frac{\sigma}{4\theta}\int_\Omega \frac{H_{\frac 1m}(|x-y|)\varphi^2}{(\frac 1m+|D\boldsymbol u_\varepsilon|^2)^{\frac{\sigma+2}{2}}}(\varepsilon+|D{\boldsymbol u}_\varepsilon|^2)^{(p-1)}\|D^2\boldsymbol{u}_\varepsilon\|^2\, dx+ C(p,\rho, \sigma, \gamma).
\end{eqnarray} 
Choosing $\displaystyle \theta <\frac{\tau}{\sigma}$ we infer that
\begin{eqnarray}\nonumber
&&\int_{B_\rho(x_0)}\frac{H_{\frac 1m}(|x-y|)}{(\frac 1m +|D\boldsymbol u_\varepsilon|^2)^\frac{\sigma}{2}}\, dx\leq \int_{\Omega}\frac{H_{\frac 1m}(|x-y|)\varphi^2}{(\frac 1m +|D\boldsymbol u_\varepsilon|^2)^\frac{\sigma}{2}}\, dx\\\nonumber
&\leq& C(\sigma,\tau)\int_\Omega \frac{H_{\frac 1m}(|x-y|)\varphi^2}{(\frac 1m+|D\boldsymbol u_\varepsilon|^2)^{\frac{\sigma+2}{2}}}(\varepsilon+|D{\boldsymbol u}_\varepsilon|^2)^{(p-1)}\|D^2\boldsymbol{u}_\varepsilon\|^2\, dx+ C(p,\rho, \sigma, \gamma).
\end{eqnarray}
As before, let $m=\frac{1}{\varepsilon}.$ We obtain
\begin{eqnarray}\nonumber
&&\int_{B_\rho(x_0)}\frac{H_{\varepsilon}(|x-y|)}{(\varepsilon+|D\boldsymbol u_\varepsilon|^2)^\frac{\sigma}{2}}\, dx\\\nonumber&\leq& C(\sigma,\tau)\int_{B_{2\rho}} {\left(\varepsilon+|D\boldsymbol u_\varepsilon|^2\right)^{\frac{\alpha-\sigma+(p-2)}{2}}}\frac{(\varepsilon+|D{\boldsymbol u_\varepsilon}|^2)^\frac{p-2-\alpha}{2}\|D^2\boldsymbol{u}_\varepsilon\|^2\varphi^2}{|x-y|^\gamma}\, dx\\\nonumber&&+C(p,\rho, \sigma, \gamma).\end{eqnarray}
Since $\sigma$ is as in \eqref{zicurli}, there exists $\alpha\geq 0$ such that   $\alpha-\sigma+(p-2)\geq 0$ and    estimation~\eqref{eq:princ1i} of Theorem \ref{teosecond} holds. We obtain 
\begin{equation}\label{eq:estimates4}
\int_{B_\rho(x_0)}\frac{H_{\varepsilon}(|x-y|)}{(\varepsilon+|D\boldsymbol u_\varepsilon|^2)^\frac{\sigma}{2}}\, dx\leq C(f,n, p, \gamma, \rho, \sigma,\tau).
\end{equation}
Using  Fatou Lemma in \eqref{eq:estimates4}  and a classical  finite covering argument,  we  get the thesis. 
\end{proof}
%\begin{remark}
% Even if $p>2$ implies $p-2+\alpha>0$,  for any choice of $\alpha$, when $p\in (1,2)$,  the above $\sigma$ is positive if
%$$
% \frac{5\sqrt{N}+1-\sqrt{N+6\sqrt{N}+1}}{2\sqrt{N}}<p<2.
%$$ 
%\end{remark}
The estimates for  $D^2\boldsymbol u$ and $|D\boldsymbol u|^{-1}$ obtained in Theorem \ref{teosecond} and  Theorem \ref{Reg2} gives us tools for proving more regularity for our weak solution $\boldsymbol u$. We have
\begin{theorem}\label{Reg}
Let $\Omega$ be a bounded  smooth domain and $p>1$. Let $\boldsymbol{u}\in C^1({\Omega})$ be a weak solution of \eqref{system1} and assume $\boldsymbol  f(x)\in W_{loc}^{1,1}(\Omega)\cap C_{loc}^{0,\beta}(\Omega).$ Then
\begin{itemize}
\item[$(i)$] if  $p>\frac 32$, recovering the result already proved in \cite{Cma}
\begin{equation}\label{eq:w12}
|D{\boldsymbol u}|^{p-2}D{\boldsymbol u} \in W^{1,2}_{loc}(\Omega);
\end{equation}
\item[$(ii)$] if  $1<p<3$,
then 
\[\boldsymbol u\in W^{2,2}_{loc}(\Omega);\]
\item[$(iii)$] if  $p \geq 3$ and for some $i=1,\dots, N$
	\[
	f^i\geq \tau >0 \quad in \quad  \Omega,
	\] 
	then 
\[\boldsymbol u\in W^{2,q}_{loc}(\Omega),\mbox{ with } 1\leq q<\frac{p-1}{p-2}.\]\end{itemize}
\end{theorem}
\begin{proof}[Proof of Theorem \ref{Reg}]	
In this proof we use the regularity results contained in Theorems \ref{teosecond}-\ref{Reg2} in the case $\gamma=0$, see Remark \ref{eq:avetecapitochecisnio}. \\
We start proving $(iii)$. Let $\boldsymbol u_\varepsilon$  be the solution to \eqref{system} with $\boldsymbol f(x):= \boldsymbol f(x,\boldsymbol u(x))$ and $\boldsymbol u$ solution to \eqref{system1}. 
\noindent We  split  $|D^2 \boldsymbol{u}_\varepsilon|^q$  as 
\[|D^2 \boldsymbol{u}_\varepsilon|^q=|D\boldsymbol{u}_\varepsilon|^{\frac{(p-2-\alpha)q}{2}}|D^2 \boldsymbol{u}_\varepsilon|^q\frac{1}{|D \boldsymbol{u}_\varepsilon|^{\frac{(p-2-\alpha)q}{2}}}.\]
Let us exploit estimation \eqref{eq:estimates4} of Theorem \ref{Reg2} together with the fact that, in this case, if $p\geq 3$, $q<2$. Then
\begin{eqnarray}\nonumber
&&\int_{B_\rho(x_0)}|D^2 \boldsymbol{u}_\varepsilon|^qdx\\\nonumber
&&= \int_{B_\rho(x_0)}(\varepsilon+|D \boldsymbol{u}_\varepsilon|^2)^{\frac{(p-2-\alpha)q}{4}}|D^2 \boldsymbol{u}_\varepsilon|^q\frac{1}{(\varepsilon+|D \boldsymbol{u}_\varepsilon|^2)^{\frac{(p-2-\alpha)q}{4}}}dx\\\nonumber
&&\leq \left(\int_{B_\rho(x_0)}(\varepsilon+|D \boldsymbol{u}_\varepsilon|^2)^{\frac{p-2-\alpha}{2}}|D^2 \boldsymbol u_\varepsilon|^2dx\right)^{\frac q2} \left(\int_{B_\rho(x_0)}\frac{1}{(\varepsilon+|D \boldsymbol{u}_\varepsilon|^2)^{\frac{(p-2-\alpha)q}{4-2q}}}dx\right)^{\frac{2-q}{2}}\\\nonumber
&&\leq C \left(\int_{B_\rho(x_0)}\frac{1}{(\varepsilon+|D \boldsymbol{u}_\varepsilon|^2)^{\frac{(p-2-\alpha)q}{4-2q}}}dx\right)^{\frac{2-q}{2}},
\end{eqnarray}
by \eqref{eq:princ1i}. Whenever  the right hand side is bounded (see Theorem \ref{Reg2} and \eqref{eq:estimates4}), namely for 
\[\frac{(p-2-\alpha)q}{4-2q}<\frac{p-1}{2},\]
choosing $\alpha\simeq 1$
since $q<\frac{p-1}{p-2}$,
we deduce that 
\[\int_{B_\rho(x_0)}|D^2 \boldsymbol{u}_\varepsilon|^qdx\leq C.\]
If $\tilde{\Omega}\subset\subset \Omega$,  by finite covering argument we get 
\[\int_{\tilde \Omega}|D^2 \boldsymbol{u}_\varepsilon|^qdx\leq C,\]
where $C$ does not depend on $\varepsilon$.
It is now easy to see that, consequently, $\boldsymbol u\in W^{2,q}_{loc}(\Omega)$ in this case. Indeed $\boldsymbol{u}_\varepsilon \rightharpoonup \boldsymbol{w} \in W^{2,q}(\tilde \Omega)$. Moreover because of the compact embedding in $L^q(\tilde \Omega)$, up to a subsequence, $\boldsymbol{u}_\varepsilon \rightarrow \boldsymbol{w}$ a.e. in $\tilde \Omega$. On the other hand (see \cite{DiKaSc})
\[
\boldsymbol {u}_\varepsilon \rightarrow \boldsymbol{u}\quad  \text{a.e.}
\]
Hence \[\boldsymbol{u}\equiv \boldsymbol{w}\in W^{2,q}(\tilde \Omega).\]
We prove now the cases $(i)$-$(ii)$.   
We set 
\begin{equation}\label{eq:ciaopiero}
	 V_{\varepsilon,\eta,i,j}(x):=(\varepsilon+|D{\boldsymbol u_\varepsilon}|^2)^\frac{\eta-1}{2}u_{\varepsilon,i}^{j}.
\end{equation}
Since
$$
\nabla V_{\varepsilon,\eta,i,j}=(\eta-1)(\varepsilon+|D{\boldsymbol u_\varepsilon}|^2)^\frac{\eta-3}{2}|D{\boldsymbol u_\varepsilon}|\nabla(|D \boldsymbol u_\varepsilon|)(u_{\varepsilon,i}^{j})+(\varepsilon+|D{\boldsymbol u_\varepsilon}|^2)^\frac{\eta-1}{2}\nabla u_{\varepsilon,i}^{j}.
$$
 We have 
$$
 \int_{B_\rho(x_0)} |D V_{\varepsilon,\eta}|^2\, dx\leq C\int_{B_\rho(x_0)}(\varepsilon+|D{\boldsymbol u_\varepsilon}|^2)^{\eta-1}|D^2\boldsymbol u_\varepsilon|^2\, dx,
$$
for a suitable positive constant $C=C(\eta)$.  \\
Using  \eqref{eq:princ1i} of Theorem \ref{teosecond}, if $p> \frac 32 $ and $\eta =p-1$ we get 
$$
 \int_{B_\rho(x_0)} |D V_{\varepsilon,\eta}|^2\, dx\leq C\int_{B_\rho(x_0)}(\varepsilon+|D{\boldsymbol u_\varepsilon}|^2)^{p-2}|D^2\boldsymbol u_\varepsilon|^2\, dx\leq C.
$$
Using this estimate, arguing exactly as in the previous case, we deduce $(i)$ 
\[
|D{\boldsymbol u}|^{p-2}D{\boldsymbol u} \in W^{1,2}_{loc}(\Omega).
\]
Finally $(ii)$ follows setting $\eta=1$ in \eqref{eq:ciaopiero}.
\end{proof}
\section{Comparison Principles and Symmetry}\label{comparison}
In this section we will first prove some comparison principles that are new for vectorial $p$-Laplace equation.
First of all,  let us state a useful  inequality; for given $\eta,\eta' \in \mathbb{R}^{q \times r}$,  there exists a positive constant $C_p=C_p(p)$ such that
\begin{equation}\label{eq:lucio}
(|\eta|^{p-2}\eta-|\eta'|^{p-2}\eta'): (\eta - \eta') \geq C_p (|\eta|+|\eta'|)^{p-2}|\eta-\eta'|^2.
\end{equation}
As a consequence of Theorem \ref{Reg2},  we deduce a weighted Poincar\'e type inequality and then we use it to get a weak comparison principle in small domains.   The integrability property \eqref{zicurli} is the key  to get weighted Poincar\'e inequality. Indeed once  \eqref{zicurli} is in force, we deduce the  Poincar\'e inequality  from \cite[Theorem 8]{FMS}. 
Let us remark that the  Poincar\'e inequality is relevant to the case $p>2$ in the proof of the symmetry result.
We have the following 
\begin{theorem}[\cite{DamSci, FMS}]\label{PoincareEsp}
 Let $p>2$, $\boldsymbol{u}\in C^{1}({\Omega})$ be a weak solution of \eqref{system1}   with  $\Omega\subset \mathbb{R}^n$  a smooth bounded domain and assume that $\boldsymbol{f}(x,\cdot)$ satisfies \eqref{hpf} $(i)$-$(ii)$.  \\
  For any $\Omega '\subset \subset \Omega,$ let $u$ be such that   
    \begin{equation*}
        \int_{\Omega'} \frac{1}{|D{\boldsymbol u}|^{(p-2)t}|x-y|^{\gamma}}\,dx\leq C,
    \end{equation*}
with $1<t<\frac{p-1}{p-2}$, $\gamma <n-2$ if $n\geq 3$  ($\gamma =0$ if $n=2$).  Then for any $w\in {C^\infty_c(\Omega')}$ 
%with respect to the norm
%\[\|w\|^2:=\int_{\Omega'} w^2\, dx+ \int_{\Omega'} |D{\boldsymbol u}|^{p-2}|\nabla w|^2\, dx,\]
there exist a constant $C_S(|\Omega '|)$, with $C_S(|\Omega '|)\rightarrow 0$ if $|\Omega '|\rightarrow 0$, such that
    \begin{equation}\label{Poincarepesata}
    \|w\|_{L^q(\Omega ')}\leq C_S(|\Omega '|)\left(\int_\Omega |D{\boldsymbol u}|^{p-2}|\nabla w|^2\right)^{\frac{1}{2}},
    \end{equation}
   for any $1\leq q< 2^*(t)$, where 
    \begin{equation*}
        \frac{1}{2^*(t)}=\frac{1}{2}-\frac 1n+\frac{1}{t}\left(\frac 12 -\frac \gamma {2n}\right).
    \end{equation*}
    Furthermore \eqref{Poincarepesata} holds for any $w\in H^1_{0,\rho}(\Omega')$, \cite{CiDeSci}.
\end{theorem}
\begin{remark}
As already noted in Remark 4 in \cite{FMS},  the largest value of $2^*(t)$ is obtained for $\gamma\simeq n-2$ ($\gamma=0$,  whenever $n=2$) and $t\simeq\frac{p-1}{p-2}$. In this case
\begin{equation*}
   \frac{1}{2^*(p)}\simeq\frac{1}{2}-\frac 1n+\frac{p-2}{p-1} \frac 1n
 \end{equation*}
and therefore $2^*(p)>2$.\\
Moreover we point out that if $\boldsymbol{u}\in C^1(\bar\Omega)$ and $Z_{\boldsymbol{u}}\subset\subset \Omega$, it follows that the estimate \eqref{Poincarepesata} holds for any $\Omega'\subseteq \Omega.$
\end{remark}
\noindent Theorem \ref{PoincareEsp} allows us to prove  \begin{proof}[Proof of Theorem \ref{thm:weak1}]
%The case $1<p\leq2$ can be obtained \textcolor{red}{following the proof of Theorem 2.3 in~\cite{MRS}}. Here we suppose $p> 2$.
Let us consider the vectorial function 
\begin{equation}\label{eq:(20)12}
(\boldsymbol{u-v})^+:= \left( (u^1-v^1)^+,\ldots, (u^\ell-v^\ell)^+,\ldots, (u^N-v^N)^+ \right).
\end{equation}
Since $(\boldsymbol{u-v})^+$ is a good test function for equations \eqref{eq:(20)0},  using \eqref{ipotesibrut} together with \eqref{eq:lucio}, we obtain  
\begin{eqnarray}\label{eq:(20)2}
&&C_p \int_{\tilde\Omega}(|D {\boldsymbol u}|+|D \boldsymbol{v}|)^{p-2}|D (\boldsymbol{u-v})^+|^2 dx \nonumber \\
&& \leq \int_{\tilde\Omega} (|D \boldsymbol{u}|^{p-2}D \boldsymbol{u} - |D \boldsymbol{v}|^{p-2}D \boldsymbol{v}: D (\boldsymbol{u-v})^+)\, dx  \nonumber \\
&& \leq \int_{\tilde\Omega}\langle(\boldsymbol{f}(x,\boldsymbol{u})-\boldsymbol{f}(x,\boldsymbol{v})), (\boldsymbol{u-v})^+\rangle\, dx.
\end{eqnarray}
Now we can evaluate the right term of \eqref{eq:(20)2} as
\begin{eqnarray*}
&& \int_{\tilde\Omega}\langle(\boldsymbol{f}(x,\boldsymbol{u})-\boldsymbol{f}(x,\boldsymbol{v})),(\boldsymbol{u-v})^+\rangle \, dx\\\nonumber &&=\int_{\tilde\Omega} \sum_{\ell=1}^N \left[f^\ell(x,u^1,\ldots,u^N)-f^\ell(x,v^1,\ldots,v^N)\right] (u^\ell-v^\ell)^+\, dx   \nonumber \\
&& = \int_{\tilde\Omega} \sum_{\ell=1}^N \left[f^\ell(x,u^1,u^2,\ldots, u^N)-f^\ell(x,v^1,u^2\ldots, u^N)+f^\ell(x,v^1,u^2\ldots, u^N) \right. \nonumber \\
&& \qquad \qquad \qquad  \left.  -f^\ell(x,v^1,\ldots,v^N)\right] (u^\ell-v^\ell)^+ \,dx.
\end{eqnarray*}
Iterating this procedure for the remaining variables we get
\begin{eqnarray}\label{eq:(20)21}
&& \int_{\tilde\Omega}\langle (\boldsymbol{f}(x,\boldsymbol{u})-\boldsymbol{f}(x,\boldsymbol{v})), (\boldsymbol{u-v})^+\rangle \, dx\\\nonumber
&=& \int_{\tilde\Omega}\sum_{\ell=1}^N  \left[f^\ell(x,u^1,u^2,\ldots, u^N)-f^\ell(x,v^1,u^2\ldots, u^N)+f^\ell(x,v^1,u^2,\ldots, u^N) \right. \nonumber \\
&& -f^\ell(x,v^1,v^2,\ldots, u^N)+\cdots+f^\ell(x,v^1,v^2,\ldots, u^\ell,\ldots, u^N)\nonumber \\ &&\qquad \qquad \qquad-f^\ell(x,v^1,v^2,\ldots, v^\ell,\ldots, u^N) \nonumber \\
&& \qquad \vdots  \nonumber  \\
&& +\ldots + \left. f^\ell(x,v^1,v^2,\ldots, u^N)-f^\ell(x,v^1,v^2,\ldots, v^N) \right] (u^\ell-v^\ell)^+\, dx \nonumber.
\end{eqnarray}
Now, having assumend among the others the cooperative condition $({h_f})$-$(iii)$, \eqref{eq:(20)21} becomes
\begin{eqnarray}\label{eq:(20)22}
&& \int_{\tilde\Omega}\langle(\boldsymbol{f}(x,\boldsymbol{u})-\boldsymbol{f}(x,\boldsymbol{v})), (\boldsymbol{u-v})^+\,\rangle dx \nonumber \\
&\leq& \int_{\tilde\Omega} \sum_{\ell=1}^N \left[   \frac{f^\ell(x,u^1,u^2,\ldots, u^N)-f^\ell(x,v^1,u^2\ldots, u^N)}{(u^1-v^1)^+} (u^1-v^1)^+ (u^\ell-v^\ell)^+ \right. \nonumber \\
&& + \frac{ f^\ell(x,v^1,u^2,\ldots, u^N)  -f^\ell(x,v^1,v^2,\ldots, u^N)}{(u^2-v^2)^+} (u^2-v^2)^+ (u^\ell-v^\ell)^+ \nonumber \\
&& \vdots    \nonumber\\
&&+  \frac{f^\ell(x,v^1,v^2,\ldots, u^\ell,\ldots, u^N)-f^\ell(x,v^1,v^2,\ldots, v^\ell,\ldots, u^N)}{ (u^\ell-v^\ell)^+}{[(u^\ell-v^\ell)^+]^2} \nonumber  \\
 && \vdots \nonumber  \\
&& + \left. \frac{f^\ell(x,v^1,v^2,\ldots, u^N)-f^\ell(x,v^1,v^2,\ldots, v^N)}{(u^N-v^N)^+} (u^N-v^N)^+ (u^\ell-v^\ell)^+ \right]\, dx \nonumber \\
&\leq& L_f \int_{\tilde\Omega} \sum_{\ell=1}^N  \sum_{j=1}^N   (u^j-v^j)^+ (u^\ell-v^\ell)^+\, dx \leq N L_f \int_{\tilde\Omega} \sum_{\ell=1}^N[(u^\ell-v^\ell)^+]^2 \,dx, 
\end{eqnarray}
where in the last inequality we used the Young's inequality, and $L_f$ is a positive constant depending on the constants $L_{\ell}$ in $(h_f)$-$(i)$.\\
If $1<p\leq 2$ we use classical Poincar\'e inequality in the right hand side of \eqref{eq:(20)22}. In this case we obtain 
\begin{eqnarray}\label{eq:quandotornicipensoio}\nonumber
\int_{\tilde\Omega} \sum_{\ell=1}^N[(u^\ell-v^\ell)^+]^2 \,dx&\leq& C_P(|\tilde \Omega|) \int_{\tilde\Omega} \sum_{\ell=1}^N |\nabla (u^\ell-v^\ell)^+|^2\,dx=C_P\int_{\tilde\Omega} |D (\boldsymbol{u-v})^+|^2 \, dx\\
&\leq& C \int_{\tilde\Omega}(|D {\boldsymbol u}|+|D \boldsymbol{v}|)^{p-2}|D (\boldsymbol{u-v})^+|^2 \, dx,
\end{eqnarray}
where $C$ is a positive constant depending on $\|D {\boldsymbol u}\|_{\infty}^{2-p}$ and $|\tilde \Omega|$, since $1<p\leq 2$. As well known $C\rightarrow 0$, as $|\tilde \Omega|\rightarrow 0$.
\\
In the case $p>2$ since either $\boldsymbol u$ or $\boldsymbol v$ is a solution to \eqref{system1},   \eqref{zicurli} and \eqref{Poincarepesata} hold,  the right hand side of  \eqref{eq:(20)22} becomes 
\begin{eqnarray}\label{eq:atecipensoio}
\nonumber
\int_{\tilde\Omega} \sum_{\ell=1}^N[(u^\ell-v^\ell)^+]^2 \,dx&\leq& C_S(|\tilde \Omega|) \int_{\tilde\Omega} \sum_{\ell=1}^N(|D {\boldsymbol u}|+|D \boldsymbol{v}|)^{p-2} |\nabla (u^\ell-v^\ell)^+|^2\,dx\\
&\leq& C_S(|\tilde \Omega|) \int_{\tilde\Omega}(|D {\boldsymbol u}|+|D \boldsymbol{v}|)^{p-2} |D (\boldsymbol{u-v})^+|^2 \, dx
\end{eqnarray}
Collecting \eqref{eq:(20)2}-\eqref{eq:atecipensoio} we deduce 
\begin{equation}\label{eq:nananananna}
\int_{\tilde\Omega}(|D {\boldsymbol u}|+|D \boldsymbol{v}|)^{p-2}|D (\boldsymbol{u-v})^+|^2 dx\leq C\int_{\tilde\Omega}(|D {\boldsymbol u}|+|D \boldsymbol{v}|)^{p-2} |D (\boldsymbol{u-v})^+|^2 \, dx,
\end{equation}
where $C=C(\boldsymbol f,p,N,|\tilde \Omega|, \|D {\boldsymbol u}\|_{\infty})$ goes to zero as $|\tilde \Omega|$ goes to zero. \\
There exists $\lambda:=\lambda(\boldsymbol f,p,N, \|D {\boldsymbol u}\|_{\infty})$ sufficiently small such that, if $|\tilde \Omega|<\lambda$, then $C<\frac 12$ in \eqref{eq:nananananna}. 
This is a contradiction  unless  $\boldsymbol u\leq \boldsymbol v$ in $\tilde \Omega$.
\end{proof}
The following  result that will be useful in the proofs of Theorem \ref{teoremaprincipale}. 
\begin{proposition}\label{nahnah}
Let $\Omega\subset\R^n$ be a bounded smooth domain and let $\boldsymbol{u},\boldsymbol{v}\in C^1(\Omega)$ such that
\begin{equation}\label{eq:telecom}
|D\boldsymbol{u}|^{p-2}D\boldsymbol{u}, |D\boldsymbol{v}|^{p-2}D\boldsymbol{v}\in W^{1,2}_{loc}(\Omega)
\end{equation}
and satisfying 
\begin{equation}\label{2}
        -{\boldsymbol \Delta}_p{\boldsymbol u}-\boldsymbol{f}(x,\boldsymbol{u})=0\quad -{\boldsymbol \Delta}_p{\boldsymbol v}-\boldsymbol{g}(x,\boldsymbol{v})=0\quad \text{in }\Omega.
    \end{equation}
    \\
We set $A:=\{x\in \Omega : \boldsymbol{u}(x)=\boldsymbol{v}(x)\}$ and suppose that \begin{equation}\label{1}
    \boldsymbol f(x,\boldsymbol u)<\boldsymbol g(x,\boldsymbol v) \quad \text{a.e. in }A.
\end{equation}  
Then $|A|=0$.
\end{proposition}
\begin{proof}
Using Stampacchia's theorem (see \cite[Theorem $6.19$]{S}) we have that $$|D\boldsymbol{u}|^{p-2}D\boldsymbol{u}=|D\boldsymbol{v}|^{p-2}D\boldsymbol{v}\quad \text{a.e. on } A.$$ By \eqref{eq:telecom} and Stampacchia's theorem we get $$\operatorname{\bf div}(|D\boldsymbol{u}|^{p-2}D\boldsymbol{u})=\operatorname{\bf div}(|D\boldsymbol{v}|^{p-2}D\boldsymbol{v})\quad \text{a.e. on } A.$$ By \eqref{2} and \eqref{1} we get the thesis.
\end{proof}
The last part of this section is devoted to prove the symmetry result Theorem \ref{teoremaprincipale}. For a real number $\lambda$, we set
\begin{equation*}%\label{eq:whawha}
	\Omega_\lambda=\Omega \cap \{x_1< \lambda\},\end{equation*}
\[x_\lambda =R_{\lambda}(x):=(2\lambda-x_1,x_2,\ldots,x_n),\]
which is the point reflected through the hyperplane
$T_\lambda =\{x\in \mathbb R^n\, :\, x_1=\lambda \}.$  
Also {define}
\[\boldsymbol{u}_{\lambda}(x):=\boldsymbol{u}(x_\lambda)=\boldsymbol{u}(2\lambda -x_1,x_2,\ldots, x_{n-1}, x_n).\]
\
We recall the $\boldsymbol{u}_{\lambda}$ solves the following equation 
\begin{equation*}%\label{riflessorisolve}
   \int_{\Omega_{\lambda}} |D{\boldsymbol u_{\lambda}}|^{p-2}(D {\boldsymbol u_{\lambda}} : D {\boldsymbol \varphi})\, dx = \int_{\Omega_{\lambda}}\langle {\boldsymbol f}(x{_\lambda},\boldsymbol u_{\lambda}) , \boldsymbol{\varphi}\, \rangle dx, 
\end{equation*}
for any $ {\boldsymbol \varphi} \in W^{1,p}_{0}(\Omega). $
Finally we define 

$$\Lambda_0:=\{a<\lambda<0: \boldsymbol u\leq \boldsymbol u_t \text{ in } \Omega_t \text{ for all } t\in (a,\lambda]\}.$$
\begin{proof}[Proof of Theorem \ref{teoremaprincipale}]
We start showing that $$\Lambda_0\neq \emptyset.$$ To prove this, let us consider $\lambda>a$ with $\lambda-a$ small enough, so that $|\Omega_\lambda|$ is small.  As remarked above and by \eqref{eq:fcresc}, we deduce that $\boldsymbol{u}_{\lambda}$ is {such that}
 \begin{equation}\label{ulambda}
-\boldsymbol{\Delta}_p \boldsymbol{u}_{\lambda} = \boldsymbol f(x_{\lambda},\boldsymbol{u}_{\lambda})>\boldsymbol f(x,\boldsymbol{u}_{\lambda}) \quad \Omega_{\lambda}.
\end{equation}
Since $\boldsymbol{u}\leq \boldsymbol{u}_\lambda$ on $\partial \Omega_\lambda$ and \eqref{ipotesisuisupporti} holds, by Theorem \ref{thm:weak1} it follows that $$\boldsymbol{u}\leq\boldsymbol{u}_\lambda \quad \text{in } \Omega_\lambda.$$
Since $\Lambda_0\neq \emptyset$, let us now set 
$$\overline{\lambda}=\sup \Lambda_0.$$ We shall show that 
$\overline \lambda=0.$ \\
To do this  assume that $\overline \lambda < 0$ and we reach a contradiction by proving that $\boldsymbol{u}\leq \boldsymbol{u}_{\overline \lambda+\tau}$ in $\Omega_{\overline \lambda+\tau}$ for any $0<\tau<\overline\tau$,  for some small $\overline \tau>0$. \\\\
  We set $A:=\{x\in \Omega_{\overline \lambda}\text{ : }\boldsymbol{u}(x)=\boldsymbol{u}_{\overline{\lambda}}(x) \}$ and let $\mathcal B\subset \Omega_{\overline \lambda}$ be an  open set such that $(A\cup Z_{\boldsymbol u}\cup Z_{{\boldsymbol u}_{\overline \lambda}})\cap \Omega_{\overline{\lambda}} \subset \mathcal{B}$. Theorem  \ref{Reg2} and Proposition \ref{nahnah} (using \eqref{eq:fcresc} with $\boldsymbol g(x, \boldsymbol v)=\boldsymbol f(x_{\overline \lambda}, \boldsymbol u_{\overline \lambda})$) guarantee that $|A|=|Z_{\boldsymbol u}| = 0$; therefore $\mathcal{B}$ will be of an arbitrarily small measure. \\
By continuity we have $$\boldsymbol{u}\leq \boldsymbol{u}_{\overline \lambda} \quad \text{in }\Omega_{\overline \lambda}.$$

\noindent We choose  a compact set $K\subset \Omega_{\overline \lambda}\setminus \mathcal{B}$  such that Theorem \ref{thm:weak1} applyes, e.g.  $|\Omega_{\overline \lambda}\setminus K|< \frac{\delta}{2}$, where $\delta=\delta(\boldsymbol f,p,N, \|D {\boldsymbol u}\|_{\infty})$.\\
We set $$S_1:=\left\{x\in K: u^1(x)<u^1_{\overline{\lambda}}(x)\right\}, \quad \tilde S_1:=\left\{x\in K: u^1(x)=u^1_{\overline{\lambda}}(x)\right\}.$$ Since $K\cap A=\emptyset$, without loss of generality we consider $S_1\neq\emptyset.$ We can select a compact set $K_1\subset S_1$ such that $S_1\setminus K_1$ has small measure.
By \eqref{ipotesisuisupporti}  and by the uniform continuity of $\boldsymbol{u}$, we deduce that 
\begin{equation}\label{ok2}
\boldsymbol{u}\leq\boldsymbol{u}_{\bar \lambda +\tau}\,\, \text{in}\,\,  K_1,
\end{equation}
for any $0<\tau<\overline \tau_1$ for some small $\overline \tau_1>0.$

Since $\tilde S_1\cap A=\emptyset$, for any $x\in \tilde S_1$ there exist at least an index $i=2,...,N$ such that $u^i<u^i_{\overline\lambda}$ in $\overline{B_r(x)}\subset\subset \Omega_{\overline{\lambda}}$, for some $r>0$.

Since $\tilde S_1$ is closed (and obviously bounded), therefore there exist compact sets $\tilde K_2\subset \left\{u^2<u^2_{\overline{\lambda}}\right\},...,\tilde K_N\subset \left\{u^N<u^N_{\overline{\lambda}}\right\}$ in $\Omega_{\overline{\lambda}}$ (some of them eventually empty) such that 
\begin{equation}\nonumber
    \tilde S_1= (\tilde S_1\cap \tilde K_2)\cup \cdot\cdot\cdot \cup(\tilde S_1\cap \tilde K_N).
\end{equation}
By \eqref{ipotesisuisupporti} and by the uniform continuity of $\boldsymbol{u}$, we have that 
\begin{equation}\label{ok}
\boldsymbol{u}\leq \boldsymbol{u}_{\bar \lambda +\tau}\,\, \text{in}\,\, \tilde S_1,
\end{equation}
for any $0<\tau<\overline \tau_2$ for some small $\tau_2>0$. 
Let us choose $0<\overline \tau< \min\{\tau_1,\tau_2\} $ sufficiently small such that $|\Omega_{\overline \lambda+ \tau} \setminus (\tilde S_1\cup K_1)| < \delta$, for any $0<\tau<\overline \tau$. \\
Since $\boldsymbol{u} \leq \boldsymbol{u}_{\overline \lambda+ \tau}$ on $ \partial (\Omega_{\overline \lambda+ \tau} \setminus  (\tilde S_1\cup K_1))$ and  \eqref{ipotesisuisupporti}, Theorem \ref{thm:weak1} permits to have
$$\boldsymbol{u} \leq \boldsymbol{u}_{\overline \lambda+ \tau}\quad \text{in } \Omega_{\overline\lambda+ \tau} \setminus (\tilde S_1\cup K_1)),$$
for any $0<\tau<\overline \tau.$
Exploiting also \eqref{ok} and \eqref{ok2}, we have that 
$$\boldsymbol{u} \leq \boldsymbol{u}_{\overline \lambda+ \tau}\quad \text{in } \Omega_{\overline  \lambda+ \tau},$$
for any $0<\tau<\overline \tau$. \\
This is a contradiction with the definition of $\overline \lambda$, hence we have $\overline \lambda=0$ and $\boldsymbol{u}\leq \boldsymbol{u}_0$ in $\Omega_0$. \\ \\
Performing the moving plane technique in the opposite direction, we deduce that $\boldsymbol{u}$ is symmetric with respect to the hyperplane $\{x_1=0\}.$ We also remark that the solution is increasing in the $x_1$-direction in $\{x_1 < 0\}$ since it is implicit  in the moving plane procedure.\\

\end{proof}

\begin{center}
{\bf Acknowledgements}
\end{center} 
L. Montoro, L. Muglia, B. Sciunzi and D. Vuono are partially supported by PRIN project 2017JPCAPN (Italy): Qualitative and quantitative aspects of nonlinear PDEs, and L. Montoro by  Agencia Estatal de Investigación (Spain), project PDI2019-110712GB-100.

\

\noindent All the authors are partially supported also by Gruppo Nazionale per l’Analisi Matematica, la Probabilit\`a  e le loro Applicazioni (GNAMPA) of the Istituto Nazionale di Alta Matematica (INdAM).

\begin{center}
	{\sc Data availability statement}\
	All data generated or analyzed during this study are included in this published article.
\end{center}

\

\begin{center}
	{\sc Conflict of interest statement}
	\
	The authors declare that they have no competing interest.
\end{center}

%\
%
%\noindent $\diamond$ The authors would really like to thank the anonymous referee for his/her useful hints and comments.


\begin{thebibliography}{99}


\bibitem{BaCiDiMa} {A. Balci, A. Cianchi, L. Diening, V. Maz'ya}. { A pointwise differential inequality and second-order regularity for nonlinear elliptic systems,} {\emph{Math. Ann.}, 383 (2022), no. 3-4, 1775--1824}.




\bibitem{ChenDiB}{Y.-Z. Chen,  E. DiBenedetto}.
{Boundary estimates for solutions of nonlinear degenerate
              parabolic systems}, {\em J. Reine Angew. Math.}, {395} (1989),  {102--131}.


\bibitem{Cma}{A. Cianchi, V. Maz'ya}. {Optimal second-order regularity for the p-Laplace system, }{\emph{J. Math. Pures Appl. (9)}, 132 (2019), 41--78}.


\bibitem{CiDeSci}{S. Cingolani, M.  Degiovanni, B.  Sciunzi}.
{Weighted {S}obolev spaces and {M}orse estimates for quasilinear elliptic equations},
{\em J. Funct. Anal.}, {286 (2024)}, no. 8, {Paper No. 110346}.


\bibitem{DamPa}{L. Damascelli, F. Pacella}. {Monotonicity and symmetry of solutions of p-Laplace equations $1<p<2$ via the moving plane method, }{\emph{Ann. Scuola Norm. Sup. Pisa Cl. Sci. (4)},  26 (1998), no. 4,  689--707.}


\bibitem{DamSci}{L. Damascelli, B. Sciunzi}. {Regularity, monotonicity and symmetry of positive solutions of $m$-Laplace equations, }{\emph{ J. Differential Equations}, 206 (2004), no. 2, 483--515.}

\bibitem{DM} C.  De Filippis, G. Mingione. {On the regularity of minima of non-autonomous functionals,} {\emph The Journal of Geometric Analysis}, 30 (2020), 1584--1626.

\bibitem{DM2} C.  De Filippis, G. Mingione. {Regularity for Double Phase Problems at
	Nearly Linear Growth,} {\emph Arch. Rational Mech. Anal.},  247(2023), no.5, Paper No. 85, 50 pp.

\bibitem{DM3} C.  De Filippis, G. Mingione. {Nonuniformly elliptic Schauder theory,} {\emph Invent. Math.},  234(2023), no.3, 1109--1196.

\bibitem{DiKaSc}{L. Diening, P. Kaplick\'{y}, S. Schwarzacher}.
\newblock {B{MO} estimates for the {$p$}-{L}aplacian}, {\em Nonlinear Anal.}, {75} (2012), no. 2,  {637--650}.

\bibitem{FMS} {A.~Farina, L.~Montoro and B.~Sciunzi}.
\newblock  Monotonicity of solutions of quasilinear degenerate elliptic equations in half-spaces.
\newblock {\em Math. Ann.}, 357 (2013), no. 3,  855--893.


\bibitem{GT} D.~Gilbarg, N.~S.~Trudinger. \emph{Elliptic Partial Differential Equations of Second Order}.
\newblock {Springer, Reprint of the 1998 Edition}.

\bibitem{S} {\sc E. H. Lieb and M. Loss}.
\newblock Analysis, volume 14 of Graduate Studies in Mathematics.
\newblock {\em American Mathematical Society}, Providence, RI, 1997.
    
\bibitem{MRS}{R. López-Soriano, L. Montoro, B. Sciunzi}. {Symmetry and monotonicity results for solutions of vectorial $p$-Stokes systems},
\newblock {\em Trans.  Amer. Math. Soc.} 376 (2023), no. 5,  3493--3514.



\bibitem{KuuMin}{T. Kuusi, G. Mingione}. {A nonlinear Stein theorem, }{\emph{Calc. Var. Partial Differential Equations}, 51 (2014), no. 1-2, 45--86.}



\bibitem{M} {M.Miśkiewicz}.{Fractional differentiability for solutions of the inhomogeneous $p$-Laplace system.,} {\em Proc.  Amer. Math. Soc.}, 146 (2018), no.7, 3009--3017.

\end{thebibliography}
\end{document}